\documentclass[hidelinks,12pt,a4paper]{article}
\usepackage[utf8x]{inputenc}
\usepackage[top=30mm, bottom=30mm, inner=20mm, outer=20mm, headsep=10mm, footskip=12mm]{geometry}
\usepackage{amssymb}
\usepackage{amsmath,bm}
\usepackage{mathtools}
\usepackage{amsthm}
\usepackage[english]{babel}
\usepackage{lmodern}
\usepackage{microtype}
\usepackage[mathscr]{euscript}
\usepackage[colorlinks=true]{hyperref}
\usepackage[usenames,dvipsnames]{xcolor}
\usepackage{todonotes}
\usepackage{appendix}
\usepackage{cleveref}
\usepackage{bm}\usepackage{cancel}
\usepackage{upgreek}
\usepackage[normalem]{ulem}

\numberwithin{equation}{section}

\theoremstyle{plain}
\newtheorem{definition}{Definition}[section]

\newtheorem{theorem}[definition]{Theorem}
\newtheorem{proposition}[definition]{Proposition}
\newtheorem{lemma}[definition]{Lemma}

\theoremstyle{definition}
\newtheorem{remark}[definition]{Remark}
\newtheorem{example}[definition]{Example}

\makeatletter
\let\save@mathaccent\mathaccent
\newcommand*\if@single[3]{%
  \setbox0\hbox{${\mathaccent"0362{#1}}^H$}%
  \setbox2\hbox{${\mathaccent"0362{\kern0pt#1}}^H$}%
  \ifdim\ht0=\ht2 #3\else #2\fi
  }
\newcommand*\rel@kern[1]{\kern#1\dimexpr\macc@kerna}
\newcommand*\widebar[1]{\@ifnextchar^{{\wide@bar{#1}{0}}}{\wide@bar{#1}{1}}}
\newcommand*\wide@bar[2]{\if@single{#1}{\wide@bar@{#1}{#2}{1}}{\wide@bar@{#1}{#2}{2}}}
\newcommand*\wide@bar@[3]{%
  \begingroup
  \def\mathaccent##1##2{%
    \let\mathaccent\save@mathaccent
    \if#32 \let\macc@nucleus\first@char \fi
    \setbox\z@\hbox{$\macc@style{\macc@nucleus}_{}$}%
    \setbox\tw@\hbox{$\macc@style{\macc@nucleus}{}_{}$}%
    \dimen@\wd\tw@
    \advance\dimen@-\wd\z@
    \divide\dimen@ 3
    \@tempdima\wd\tw@
    \advance\@tempdima-\scriptspace
    \divide\@tempdima 10
    \advance\dimen@-\@tempdima
    \ifdim\dimen@>\z@ \dimen@0pt\fi
    \rel@kern{0.6}\kern-\dimen@
    \if#31
      \overline{\rel@kern{-0.6}\kern\dimen@\macc@nucleus\rel@kern{0.4}\kern\dimen@}%
      \advance\dimen@0.4\dimexpr\macc@kerna
      \let\final@kern#2%
      \ifdim\dimen@<\z@ \let\final@kern1\fi
      \if\final@kern1 \kern-\dimen@\fi
    \else
      \overline{\rel@kern{-0.6}\kern\dimen@#1}%
    \fi
  }%
  \macc@depth\@ne
  \let\math@bgroup\@empty \let\math@egroup\macc@set@skewchar
  \mathsurround\z@ \frozen@everymath{\mathgroup\macc@group\relax}%
  \macc@set@skewchar\relax
  \let\mathaccentV\macc@nested@a
  \if#31
    \macc@nested@a\relax111{#1}%
  \else
    \def\gobble@till@marker##1\endmarker{}%
    \futurelet\first@char\gobble@till@marker#1\endmarker
    \ifcat\noexpand\first@char A\else
      \def\first@char{}%
    \fi
    \macc@nested@a\relax111{\first@char}%
  \fi
  \endgroup
}
\makeatother

\newcommand{\R}{\mathbf R}

\newcommand{\N}{\mathbf N}

\renewcommand{\d}{\,\mathrm{d}}

\newcommand{\eps}{\varepsilon}

\usepackage[xcolor]{changebar}
\cbcolor{blue}

\newcommand{\limi}{\varliminf}
\newcommand{\lims}{\varlimsup}

\usepackage[inline]{enumitem}
\newcommand{\enumlabelformat}{\roman}

\newlength{\thelabelsep}
\newcounter{inlineenum}
\renewcommand{\theinlineenum}{\enumlabelformat{inlineenum}}
\let\epsilon\varepsilon
\let\phi\varphi

\newcommand{\Prob}{\mathscr{P}}

\DeclareMathOperator{\supp}{spt}

\newcommand{\bdpi}{\boldsymbol{\pi}}

\newcommand{\eval}{\mathsf{e}}
\newcommand{\CC}{\mathrm{LCC}}

\newcommand{\mres}{\mathbin{\vrule height 1.6ex depth 0pt width
0.13ex\vrule height 0.13ex depth 0pt width 1.3ex}}

\newcommand{\LCC}{\mathrm{LCC}}

\newcommand{\e}{\eval}
\newcommand{\nchi}{{\raise.3ex\hbox{$\chi$}}}
\newcommand{\Lmax}{{\mathsf L}_{\mathsf{max}}}
\newcommand{\sta}{{\mathsf {st}}_{a}}

\usepackage[runin]{abstract}

\makeatletter
\let\@fnsymbol\@arabic
\makeatother

\allowdisplaybreaks

\title{PDE aspects of the dynamical optimal transport in the Lorentzian setting} %The Lorentzian Benamou--Brenier formula}
\author{Nicola Gigli, Felix Rott, Matteo Zanardini \footnote{SISSA, ngigli@sissa.it, frott@sissa.it, mzanardi@sissa.it}}
\date{\today}

\begin{document}

\maketitle
\begin{abstract}
One of the crucial features of optimal transport on Riemannian manifolds is the equivalence of the `static', original, formulation of the problem and of the `dynamic' one, based on the study of the continuity equation. This furnishes the key link between Wasserstein geometry and PDEs that has found so many applications in the last 20 years.

In this paper we investigate this kind of equivalence on spacetimes. At the PDE level, this requires to transition from the continuity equation to a suitable `continuity inequality', to which we shall refer to as `causal continuity inequality'.

As a direct consequence of our findings we obtain a Lorentzian version of the celebrated Benamou--Brenier formula.
\end{abstract}
\tableofcontents

\section{Introduction} 
Let us recall the standard link between $W_p$-absolutely continuous curves of measures on a Riemannian manifold and weak solutions of the continuity equation as established in \cite{AmbrGigSav} (strictly speaking, in \cite{AmbrGigSav} only the case $M=\R^d $ has been considered, but up to a Nash embedding the general case is then quite straightforward).
\begin{theorem}\label{thm:BBRiem}
Let $M$ be a smooth complete Riemannian manifold, $p\in(1,+\infty)$ and $(\mu_t)\subset \Prob_p(M)$ be a weakly continuous curve. Then:
\begin{itemize}
\item[\emph{(A)}]  Suppose that $(\mu_t)$ is $W_p$-absolutely continuous. Then there are vector fields $(v_t)$ so that
\begin{equation}
\label{eq:CE}\tag{CE}
\partial_t\mu_t+{\rm div}(v_t\mu_t)=0\quad\text{in the sense of distributions}
\end{equation}
and moreover
\[
|\dot\mu_t|\geq \|v_t\|_{L^p_{\mu_t}}\qquad a.e.\ t\in[0,1],
\]
where $|\dot\mu_t|:=\lim_{h\to 0}\frac{W_p(\mu_t,\mu_{t+h})}{h}$ (the a.e.\ existence of $|\dot\mu_t|$ follows from absolute continuity).
\item[\emph{(B)}] Suppose that $(v_t)$ is a family of vector fields so that \eqref{eq:CE} holds and $\int_0^1\|v_t\|_{L^p_{\mu_t}}\,\d t<+\infty$. Then $(\mu_t)$ is $W_p$-absolutely continuous and 
\[
|\dot\mu_t|\leq \|v_t\|_{L^p_{\mu_t}}\qquad a.e.\ t\in[0,1].
\]
\end{itemize}
\end{theorem} 
This statement is crucial in modern optimal transport as it links Wasserstein geometry to PDEs. Among other things, it can be used to give a rigorous foundation to Otto's calculus and can be used to justify the celebrated Benamou--Brenier formula.

In this paper we are interested in deriving a Lorentzian analogue of the above, where the manifold $M$ is now a smooth globally hyperbolic spacetime. In doing so it is rather natural to replace $W_p$-absolute continuity with causality of the curve of measures $(\mu_t)$. A central question is then how to interpret the continuity equation \eqref{eq:CE} in this setting. It turns out that it is appropriate to replace it with the `causal continuity inequality'. Loosely speaking, and referring to the body of the paper for the rigorous definition,  given a causal curve of measures $(\mu_t) $ and a family $(v_t)$ of Borel and future directed vector fields, we say that these satisfy the causal continuity inequality provided
\[
\frac{\d}{\d t}\int \varphi\,\d\mu_t\geq \int\d\varphi(v_t)\,\d\mu_t\qquad a.e.\ t
\]
for every $\varphi:M\to\R$ smooth causal function. With this in mind, our main result is: 
\begin{theorem} \label{Thm: main}
Let $(M,g)$ be a globally hyperbolic spacetime, $0 \neq p < 1$ and $(\mu_t)\subset  \Prob(M)$ be a causal path such that $\supp( \mu_t) \subset K$ for every $t \in [0,1]$, for some compact set $K\subset M$. Then:
\begin{itemize}
\item[\emph{(A)}] There exists a family of causal Borel vector fields $(v_t)$ such that $(\mu_t, v_t)$ satisfies  the causal continuity inequality and moreover
\begin{equation}
\int_0^1\tfrac1p |\dot \mu_t|_p ^p\,\d t \leq \int_0^1\tfrac1p \| v_t \|_{L^p_{\mu_t}} ^p\,\d t
\end{equation}
for almost every $t \in [0,1]$, where $|\dot\mu_t|_p = \lim_{h\downarrow0}\frac1h \ell_p(\mu_t,\mu_{t+h})$ (the a.e.\ existence of $|\dot\mu_t|_p$ follows from the causality of the curve -- see \cite[Theorem 2.23]{Octet}). 
\item[\emph{(B)}] For any family $(v_t)_{t \in [0,1]}$ of causal Borel vector fields  such that $\int_0^1 \tfrac1p \|v_t\|^p_{L^p_{\mu_t}} \, \d t >-\infty$ and  $(\mu_t, v_t)$ satisfies the causal continuity inequality we have 
\begin{equation} \label{eq: norm inequality}
\tfrac1p |\dot \mu_t|^p_p  \geq \tfrac1p   \| v_t \|^p_{L^p_{\mu_t}}  \qquad a.e.\ t.
\end{equation}
\end{itemize}
\end{theorem}
Out of this, it will be easy to obtain a Lorentzian version of Benamou--Brenier formula, see Theorem \ref{thm: BB theorem Lor}. To the best of our knowledge, the only analogous result in this direction has been obtained in \cite{MillerCausal}, under quite stronger assumptions.

We comment on the proof of Theorem \ref{Thm: main} above.

For part $(A)$ we rely on the Lifting Theorem of causal paths established in \cite[Theorem 2.43]{Octet}: the desired vector fields will then be built by suitably averaging, using the lifting, the speed of curves. We remark that this, all in all very natural, procedure would work also in Riemannian signature, but as far as we know this observation does not appear previously in the literature (at least not explicitly, but see \cite[Thm 2.3.24 and Rmk 2.3.26]{Gigli14}).

For part $(B)$ we argue instead via (the Lorentzian version of) Kuwada's Lemma from \cite{GigKuwOhta}. This procedure was initially designated to translate PDE information into metric ones in the very abstract setting of Alexandrov spaces and since then has become a cornerstone of the theory of ${\sf RCD}$ spaces (see e.g.\ \cite{Degiorgiandgromov} for references). Still in the elliptic signature, this line of thought has been used in \cite{AmbrosioBrueSemola} (see also \cite{maggiOT}) to prove part $(B)$ of Theorem \ref{thm:BBRiem}: some technical aspects of our argument are inspired by these references.

\section{Preliminaries}

\subsection{Lorentzian optimal transport}

In this section we collect a few important results and citations from measure theory, optimal transport and Lorentzian geometry. 

We will often make use of the Disintegration Theorem, a fundamental result from measure theory, a proof of which can be found in, e.g., \cite{AmbrosioFuscoPallara, Bogachev, Fremlin4}. 
Notice that our setting of spacetimes -- smooth and time-oriented Lorentzian manifolds -- is considerably less general than the cited references.

%In our context, a function $f : M \rightarrow \R$ is of local bounded variation, denoted by $f \in \mathsf{BV}_{loc}(M)$, provided $f \in L^1_{loc}(M)$ (with respect to the canonical volume measure $\text{vol}_g$) and its derivative in the sense of distribution, denoted by $Df$, has components in local charts that are Radon measures. 
%\medskip

Concerning Lorentzian geometry, we refer to \cite{BeemBook, Min19} for an extensive review of basic concepts. 
We will work with the signature convention $(+,-,\ldots, -)$, which is not the one used in \cite{BeemBook, Min19}. 
In fact, our convention is less used in Lorentzian geometry in general, but we prefer to use this one as it seems more suitably tailored towards our setting: among other things, we have that the gradient of causal functions is future-directed, we do not need to put absolute values on otherwise negative scalar products, and various other sign subtleties simplify in our favor. 
On the other hand, one of the arguments for the other signature convention, namely that a spacelike slice has a truly Riemannian, i.e.\ positive, signature, is of little relevance to us. 

We will use $\leq$ and $\ll$ to denote the causal and timelike relations on a spacetime $(M,g)$, respectively, which say when two points can be connected by a piecewise $C^1$ curve $\gamma$ such that $g_{\gamma_t}(\dot \gamma_t, \dot \gamma_t) \geq 0$ (resp.\ $>0$) for all $t$. 
In the smooth setting, such curves are usually referred to as causal and timelike curves. 
We will further use the notation $J^+(x) = \{ y \in M : x \leq y \}$, and the dual notation of $J^-(x)$. 
By $I^{\pm}(x)$ we will denote the corresponding sets associated to $\ll$.
%We will use the symbol $\leq$ for the causal order induced by a given Lorentzian metric, and 
%Unless explicitly mentioned otherwise, causal curves are assumed to be future-directed. 

The following convention will be useful to us: given a sequence $(x_n)_n$ in $M$, we write $x_n \uparrow x$ if $x_n \to x$ and $x_n \leq x_{n+1} \leq x$ for every $n$. Dually, we will use $x_n \downarrow x$. 

\begin{definition}[Global hyperbolicity]
A spacetime $(M,g)$ is called globally hyperbolic if $\leq$ is a partial order and causal diamonds $J(x,y): =J^+(x) \cap J^-(y)=\{z \in M : x \leq z \leq y \}$ are compact for all $x,y \in M$. 
\end{definition}

Note that under the assumption of global hyperbolicity, we have that `causal emeralds' 
\[
J(A,B)=\{x \in M : \exists\, a \in A, b \in B: a \leq x \leq b \}
\]
are compact whenever $A,B \subset M$ are compact, see \cite[Subsection 6.6]{HawEll}.

Given $z \in M$ and $v \in T_zM$, we define
\begin{equation}
\| v \|_g \coloneqq 
\begin{cases}
\sqrt{g_z(v,v)} & \text{if $v \in J^+(0) \subset T_zM$,} \\
- \infty  & \text{otherwise} \, .
\end{cases}
\end{equation}
Then the time separation function on $(M,g)$ is denoted by $\ell$ and is defined as follows: 
\begin{equation} \label{eq: time sep}
\ell(x,y) := \sup \int_0^1 \|\dot \gamma_t \|_g \d t \, ,
\end{equation} 
where the supremum is taken over all piecewise $C^1$ future-directed causal curves connecting $x$ and $y$ (with the usual convention of $\sup \emptyset = -\infty$). 

Recall that for globally hyperbolic spacetimes it is well known that $\max\{\ell,0\}$ is finite and continuous, cf.\ \cite[Theorem 4.124]{Min19}. 
%We denote the Minkowski metric on $\R^{1,n}$ by $\langle \cdot, \cdot \rangle$ and the associated `hyperbolic norm' by $\| \cdot \|$. 
%Whenever we deal with Riemannian inner products or norms, we will explicitly decorate them with an $h$, so this will not lead to any confusion. 

We also want to mention the concept of a dual hyperbolic norm present on any cotangent space $T^*_xM$, as for example discussed in the first section of the appendix in \cite{Octet} (see also \cite{HBS}). 
Indeed, if $x \in M$ and $\omega \in T^*_xM$ is causal in the sense that $\omega(v) \geq 0$ whenever $v \in J^+(0) \subset T_xM$, then 
\begin{equation}
\label{eq: dual hyperbolic norm}
\| \omega \|_{g^*} \coloneqq \inf_{\|v\| \geq 1} \omega(v) \, .
\end{equation}
Notice the immediate implication $\omega(v) \geq \|\omega\|_{g^*} \|v\|_g$. 
We further recall that $v \in J^+(0) \subset T_xM$ if and only if $g_x(v,w) \geq 0$ for all $w \in J^+(0)$. 

We now introduce the central class of functions that appear in this work. 

\begin{definition}[Causal/steep functions and causal vector fields]
Let $(M,g)$ be a globally hyperbolic spacetime. 
We say that a function $f : M \rightarrow \R$ is causal if it is $\leq$ non-decreasing, i.e. if $x \leq y$ implies $f(x) \leq f(y)$. 

We say that $f$ is $L$-steep provided $f(y) - f(x) \geq L \ell (x,y)$ for every $x,y \in M$.
The largest such $L > 0$ is called the steepness constant of $f$ and is denoted by $\mathbf{st}(f)$.
We say that $f$ is steep if it is $L$-steep for some $L >0$. 

Moreover, a vector field $v : M \rightarrow TM$ is called causal if $v(x)$ is a future directed causal vector in $T_xM$ for every $x \in M$. 
\end{definition} 
Notice that we regard $0$ as a causal vector. 
Different from the above introduced terminology in the smooth setting, for us a causal curve is defined simply by $\gamma_s \leq \gamma_t$ for all $s \leq t$, and in particular need not to be continuous. 

We refer to \cite[Lemma 3.2, Theorems A.1 \& A.2]{Octet} for some useful properties of causal functions (and causal curves). 
Observe that if $\gamma$ is a causal curve and $f$ is a causal function, then $f \circ \gamma$ is a non-decreasing function in the classical sense. We will denote by $\LCC([0,1]; M)$ the set of all left-continuous causal curves $\gamma:[0,1] \rightarrow M$. 
This space is defined in \cite{Octet} and it is endowed with a Polish topology \cite[Proposition 2.29]{Octet}. 
Moreover, for $t \in [0,1]$ we define the Borel map $\e_t : \LCC ([0,1]; M) \rightarrow M$ by $\e_t(\gamma) := \gamma_t$, i.e. the evaluation map at $t$.

%Next, we are going to slightly expand on the notion of (local) steepness constant introduced in \cite[Section 3.1]{Octet}.

\begin{definition}[Asymptotic steepness constant]
Let $(M,g)$ be a globally hyperbolic spacetime and let $f: M \rightarrow \R$ be a causal function. 
% % Then we define the steepness constant of $f$ by
% % \begin{equation} \label{def: steepness constant}
% % \mathbf{st}(f) = \inf \bigg \{\frac{f^+(y)-f^-(x)}{\ell(x,y)} : x \ll y \bigg \} \, ,
% % \end{equation}
% % where 
% Introducing the notation
% \begin{equation} 
% \label{def: f^+ and f^-}
% f^+(y) := \inf \big \{ f(z) : \; y \ll z \in M \big \}, \qquad f^-(y) := \sup \big \{ f(x): \; M \ni x \ll y \big \}
% \end{equation}
% and
% \[
% | \partial ^+ f| (x) := \varliminf_{y \downarrow x} \frac{f^+(y)-f^-(x)}{\ell(x,y)}, \qquad | \partial ^- f| (x) := \varliminf_{z \uparrow x} \frac{f^+(x)-f^-(z)}{\ell(z,x)} \, ,
% \]
% we define the local steepness constant of $f$ at $x \in M$ by
% \begin{equation} 
% \label{def: local steepness constant}
% \mathbf{st}_{loc} (f)(x) := \min \left \{ | \partial ^+ f| (x), |\partial ^-f| (x) \right \} \, .
% \end{equation}
% %where
% %Finally, w
We define the asymptotic steepness constant by
\begin{equation} 
\label{eq: asymptotic steepness constant}
\mathbf{st}_a (f)(x) : = \sup_{x \in U \subset M} \mathbf{st} \left ( f \big{|}_{U} \right ) \, ,
\end{equation}
where the supremum is considered among all open neighborhoods $U$ of $x$ in $M$. 
\end{definition}

%\begin{remark}[Hierarchy of steepness constants]
It is easy to check that, if $f: M \rightarrow \R$ is a causal function,  it holds 
\[
0 \leq \mathbf{st}(f) \leq \mathbf{st}_a(f)(x)\qquad\forall x\in M.
\]
Moreover, notice that $\mathbf{st}_a(f): M \rightarrow [0, \infty]$ is a lower semicontinuous function. 
%\hfill$\blacksquare$
%\end{remark}

For $0 \neq p < 1$, we define the concave non-decreasing function $u_p : \overline{\R} \rightarrow \overline{\R}$ by
$u_p(z) := \frac{1}{p} z^p$ for $z > 0$ and by 
\begin{equation}
\label{eq:up}
\begin{array}{rll}
u_p(z)&=-\infty\qquad &\text{for $z\in[-\infty,0)$ and any $p\leq 1$} \, ,\\
u_p(0)&=-\infty\quad\text{(resp.\ 0)}\qquad &\text{for  $p< 0$ (resp.\ $p\in(0,1]$)} \, ,\\
u_p(+\infty)&=\phantom{-}0\quad\text{(resp.\ $+\infty$)}\qquad &\text{for  $p< 0$ (resp.\ $p\in{(}0,1]$)} \, .
\end{array}
\end{equation}
Moreover, given a causal Borel vector field $v : M \rightarrow TM$, a probability measure $\mu$ and $0 \neq p <1$ we define
\[
u_p \left ( \| v \|_{L^p_{\mu}} \right ) := \int_M u_p ( \|v(x) \|_g) \d \mu(x) \, .
\]
The integral is well defined since $\mu$ is a probability measure and $u_p(\|v\|_g) \geq 0$ if $p \in (0,1)$ as well as $u_p(\|v\|_g) \leq 0$ if $p < 0$. 

We turn to the following key definition:
\begin{definition}[Causal speed]
Let $(M,g)$ be a spacetime and let $\gamma$ be a causal curve in $M$. 
Then its causal speed is defined as 
\begin{equation}
\label{eq:speedlimit}
|\dot\gamma_t|=\lim_{h\downarrow0}\frac{\ell(\gamma_t,\gamma_{t+h})}{h}=\lim_{h\downarrow0}\frac{\ell(\gamma_{t-h},\gamma_{t})}{h}\qquad a.e.\ t.
\end{equation}
\end{definition}
The a.e.\ existence of the limit has been established, together with other basic properties, in \cite[Theorem 2.23]{Octet} on more general spacetimes, including in particular that of probability measures on a given globally hyperbolic spacetime.

In our context, notice that it holds true that $\ell(x,y) = \sup \int |\dot \gamma_t| \d t$, where the supremum is taken over any causal curve inside $J(x,y)$, that is, neither connecting $x$ and $y$ nor necessarily continuous. 
Indeed, almost everywhere, causal curves are differentiable and $\| \dot \gamma_t\|=|\dot \gamma_t|$, cf.\ \cite[Theorem A.1]{Octet}. 
Moreover, \cite[Equation (2.27)]{Octet} reads $\int_s^t |\dot \gamma_r| \d t \leq \ell(\gamma_s,\gamma_t)$ for all $s<t$ and any causal curve. 
Finally, any causal curve inside $J(x,y)$ can be extended to a causal curve from $x$ to $y$ by definition, and any such extension does not decrease its length. 

\begin{definition}[Causal transport plans]
Let $(M,g)$ be a globally hyperbolic spacetime. 
Given probability measures $\mu,\nu\in\mathscr{P}(M)$, we denote by $\Pi_{\leq}(\mu,\nu)$ the set of causal couplings, i.e.\ transport plans which are concentrated on causally related points. 
More precisely, $\pi \in \Pi_{\leq}(\mu,\nu)$ if $\mu = (\Pr_1)_{\#}\pi, \nu=(\Pr_2)_{\#}\pi$ and $\pi(\{\ell \geq 0\})=1$, where $\Pr_1,\Pr_2: M \times M \rightarrow M$ are the natural projections on the first and second component, respectively. 
We further introduce a causal relation $\preceq$ on $\mathscr{P}(M)$ via $\mu \preceq \nu$ if and only if $\Pi_{\leq}(\mu,\nu) \neq \emptyset$. 
\end{definition}

\begin{definition}[$p$-Lorentz--Wasserstein distance]
\label{Def:Lorwass} 
Let $(M,g)$ be a spacetime and let $0 \neq p <1$. 
The $p$-Lorentz--Wasserstein distance $\ell_p(\mu,\nu)$ between $\mu,\nu\in\Prob(M)$ is defined to be $-\infty$ provided $\Pi_{\leq} (\mu_0, \mu_1) = \emptyset$ and otherwise by that element in $[0,+ \infty]$ such that
\begin{equation} \label{eq: p- wassestein}
u_p \left (\ell_p(\mu,\nu) \right ):= \sup_{\pi\in \Pi_{\leq}(\mu,\nu)}  \int u_p \left ( \ell(x,y ) \right ) \, \d \pi(x,y) \, .
\end{equation}
\end{definition}

% There are a lot more of the definitions and ideas developed in \cite{Octet} (and earlier works) that we will make use of here. 
% In order to not bloat this section, we assume some familiarity with \cite{Octet}. 
% Nevertheless, the most important concepts we need are: causal transport plans, the $p$-Lorentz--Wasserstein distance, left-continuous causal paths and their causal speed and $p$-action functional $\mathcal{A}_p$ \cite[Definitions 2.12, 2.25, 2.26 \& Equation (2.30)]{Octet}. 
% Notice that for us, the causal speed of a left-continuous causal path refers to the absolutely continuous part of the Lebesgue decomposition of the measure which is originally introduced in \cite[Definition 2.25]{Octet}. 
% Some of these citations are not the earliest, but it is more convenient to have them all from the same source. 
In general, we will stick to the notation in \cite{Octet}.  

A central tool for our arguments is the Lifting Theorem for causal paths of probability measures, cf.\ \cite[Theorem 2.43]{Octet}. 
%Due to its essence for this work, we will state it explicitly (adapted to our setting). 

\begin{theorem}[Lifting paths of measures to measures on paths]
\label{Th:Lifting} 
Let $(M,g)$ be a globally hyperbolic spacetime and let $0 \neq p < 1$. 
Let $K \subset M$ be a compact subset and $[0,1] \ni t \mapsto \mu_t \in \Prob (M)$ be a left-continuous causal path with $\supp(\mu_t) \subset K$ for all $t \in [0,1]$. 
Then there exists $\bdpi\in \mathscr{P}(\CC([0,1];M))$ that induces the causal path $t \mapsto \mu_t$, i.e., $(\e_t)_{\#} \bdpi = \mu_t$ for all $t \in [0,1]$, and
\[
\int_0^1 u_p(|\dot\mu_t|)\,\d t = \iint_0^1 u_p \left (\| \dot \gamma_t \|_g \right ) \, \d t \, \d \bdpi (\gamma) \, .
\]
Moreover, if these are   $> -\infty$, then $  u_p \left ( | \dot\mu_t |_p \right )= \int u_p \left ( \| \dot\gamma_t\|_g \right ) \d \bdpi(\gamma)$ holds for a.e.\ $t \in [0,1]$.
\end{theorem}

\subsection{The Kantorovich Duality}
In this subsection we formulate a version of the Kantorovich Duality that is suitable for our setting. 
Our cost function is $c_p:=u_p\circ\ell$ for $0 \neq p <1$. We shall frequently use the fact that $c_p$ is continuous from $\{\ell\geq 0\}$ to $\R\cup\{-\infty\}$. Some points of our theorem can be extracted from the more general framework of \cite[Theorems 4.1 \& 5.9]{villanibook}, others  are also already presented in \cite{CavMon}. 

The notions of $c_p$-concave functions and $c$-transform in Lorentzian signature have been introduced in \cite{CavMon} and also appear in \cite{Octet}, however we will state them in our setting for more immediate compatibility. 

\begin{definition}[$c_p$-concave and transform]
Let $(M,g)$ be a globally hyperbolic spacetime and let $0 \neq p < 1$. 
Let $E \subset M$ be a compact subset. 
A Borel function $\varphi: E \rightarrow \overline{\R}$ is called $c_p$-concave with respect to the compact set $E$ provided there exists a function $\psi: E \rightarrow \overline{\R}$ such that
\begin{equation} \label{eq: ellp concave}
    \varphi(x) = \inf_{y \in E\cap J^+(x)} \,\left \{ \psi (y) - u_p (\ell (x,y)) \right \} \, . \notag
\end{equation}
Moreover, the $c_p$-transform with respect to $E$ of $\varphi$ is defined by $\varphi^{c_p}_E : E \rightarrow \overline{\R}$ and 
\begin{equation} \label{eq: c transform}
\varphi^{c_p}_E (y) := \sup_{x \in E\cap J^-(y)} \, \left \{\varphi (x) + u_p (\ell(x,y)) \right \} \, .
\end{equation}
As in the usual optimal transport theory, we need to slightly adapt the definition of $c_p$-transform when we consider a function whose input `lives on the first space'. 
By slight abuse of notation, let us define
\begin{equation} \label{eq: c transform2}
\psi^{c_p}_E (x) := \inf_{y \in E} \, \left \{\psi (y) - u_p (\ell(x,y)) \right \} \, .
\end{equation}
\end{definition}
Since we are interested in the dynamical optimal transport between two probability measures with compact support, we can restrict our study to a large compact set in $M$. 
To this end, we need to find a well behaved neighborhood of the support of the measures. 
The following lemma demonstrates how to construct such a neighborhood. 

\begin{lemma} 
\label{lem: finding emralds}
Let $(M,g)$ be a globally hyperbolic spacetime and let $K \subset M$ be a compact subset. Then there exists a compact emerald $E$ such that $K \subset E$ and $E$ is a neighborhood of $K$. 
Moreover, for every $y \in K$ we have  $\overline{I^-(y) \cap E} = J^-(y) \cap E$.
\end{lemma}

\begin{proof}
    Since $M$ is locally compact there exists a compact neighborhood $\tilde{K}$ of $K$. In particular, we want to find a compact emerald that contains $\tilde{K}$ with the desired properties. This last claim is a consequence of the celebrated metric Splitting Theorem by Geroch \cite{GerochDomDep} and Bernal-Sanchez \cite{BernSanchez}. Indeed, $M$ is smoothly foliated by Cauchy surfaces $(\Sigma_t)_{t \in \R}$, and moreover each $\Sigma_t$ can be assumed to be acausal, see \cite{Minguzzisteep}. Since $\tilde{K}$ is compact, we have that there exists $t_1 <t_2$ such that $\tilde{K} \subset I(\Sigma_{t_1}, \Sigma_{t_2})$, moreover $\{ I^-(a) : a \in \Sigma_{t_2} \}$ an open covering of $\tilde{K}$ and we can find $A: = \{ a_1, \dots , a_k \} \subset \Sigma_{t_2}$ such that $\tilde{K} \subset I^-(A) =I^-(a_1) \cup \dots \cup I^-(a_k)$. We define $E: = J^-(A) \cap J^+(\Sigma _{t_1})$. For every $i = 1 \dots k$, we have that $J^-(a_i) \cap J^+(\Sigma_{t_1})$ is compact because of \cite[Proposition 5.20]{penroseBook} and so $E$ is compact. \\
    It is easy to verify that $E$ is a neighborhood of $K$ and, in particular, for every $y \in K$ we have that $I^-(y) \cap E \neq \emptyset$. Moreover, $\overline{I^-(y) \cap E} \subset J^-(y) \cap E$ follows by global hyperbolicity. Vice versa, let $ z \in J^-(y) \cap E$, then there exists $a_i \in A$ such that $z \in I^-(a_i) \subset E$ by construction. Let $t_2 > r \geq t_1$ such that $z \in \Sigma_r$. Since $z \in J^-(y)$, there exists some $(z_n) _n \subset I^-(y)$ such that $z_n \rightarrow z$. Suppose that definitely
    $z_n \not \in E$ (otherwise the statement is trivial), then we can just connect $z_n$ with $y$ with timelike geodesics that definitely must intersect the Cauchy hypersurface $\Sigma _r$ in $p_n \in \Sigma_r \cap E \cap I^-(y)$. Then $p_n$ must admit a (non-relabeled) convergent subsequence $p_n \rightarrow p \in \Sigma_r$. However, we have $p_n \geq z_n \rightarrow z$ and so $p \geq z$. As $\Sigma_r$ is acausal it follows that $p=z$.
\end{proof}

The following lemma is a structural result for the $c_p$-transform. We will need this result in the proof of the Kantorovich Duality Theorem \ref{thm: kantorovich duality}.

\begin{lemma} \label{lem: basic properties of c-transform}
Let $(M,g)$ be a globally hyperbolic spacetime and let $0 \neq p <1$. Let $K \subset M$ be a compact subset and let $\varphi: M \rightarrow \R$ be a continuous and bounded function. Then there exists a compact emerald $E$ such that $E$ is a neighborhood of $K$ and $\varphi^{c_p}_E : E \rightarrow \R \cup \{ - \infty \}$ is a bounded, continuous and causal function on $K$.
\end{lemma}
\begin{proof}
We pick $E$ as in  Lemma \ref{lem: finding emralds}. The fact that $\varphi^{c_p}_E$ is bounded is easy to check (see e.g.\ the arguments in point (i) of Lemma \ref{lem: properties of Hopf--Lax}).  
It is clear that $\varphi^{c_p}_E$ is causal by definition, moreover it is finite on $K$. Since $K$ is compact we are left to prove the continuity property.
Let $y \in K$. By Lemma \ref{lem: finding emralds}, we have $\overline{I^-(y) \cap E} = J^-(y) \cap E$ and thus in particular
\begin{equation}
\label{eq:equivvc}
\varphi^{c_p}_E(y)
 = \sup_{J^-(y) \cap E} \{ \varphi(x) + u_p( \ell(x,y)) \}= \sup_{I^-(y) \cap E} \{ \varphi(x) + u_p( \ell(x,y)).
\end{equation}
Let $(y_n)_n \subset K$ such that $y_n \rightarrow y$ in $K$. Let $x\in I^-(y)\cap E$ and notice that for $n\in\N$ sufficiently big we have $x\in I^-(y_n)\cap E$ as well, therefore 
\[
\limi_n\varphi^{c_p}_E(y_n)\geq \limi_n\varphi(x)+c_p(x,y_n)=\varphi(x)+c_p(x,y).
\]  
Taking the sup in $x\in I^-(y)\cap E$ and recalling \eqref{eq:equivvc} we get lower semicontinuity of $\varphi^{c_p}_E$. For upper semicontinuity we notice that 
 for every $n \in \N$ the set   $J^-(y_n) \cap E$ is compact and so continuity of $\varphi$ and $c_p$ (on $\{\ell\geq 0\}$) grant the existence of  $ x_n \in J^-(y_n)\cap E $ such that
\[
\varphi^{c_p}_E(y_n) = \varphi(x_n) + u_p (\ell(x_n , y_n)).
\]
Since   $(x_n)_n \subset E$, passing to a non-relabeled  subsequence we can assume that $x_n\to \overline{x} \in E$. Global hyperbolicity gives the closure of $\{\ell\geq 0\}$ and thus $x_{n} \leq y_{n}$ for every $n$ implies  $\overline{x}\leq y$. It follows that
\[
\varphi^{c_p}_E (y) \geq \varphi(\overline{x}) + u_p (\ell(\overline{x},y)) = \lim_n \varphi(x_{n}) + u_p (\ell(x_{n}, y_{n})) \geq \lims_n \varphi^{c_p}_E(y_{n}),
\]
as desired.
\end{proof}

\begin{theorem}[Kantorovich Duality]
\label{thm: kantorovich duality}
Let $(M,g)$ be a globally hyperbolic spacetime and let $0 \neq p <1$. Consider $\mu, \nu \in \Prob (M)$ with compact support such that $\mu \preceq \nu$. Then:
\begin{itemize}
\item[i)] There exists a $p$-optimal transport plan $\pi \in \Pi_{\leq} (\mu, \nu)$ and $\ell_p( \mu, \nu) < \infty$.
\item[ii)] There exists a compact neighborhood $E$ of the compact set $K:=J(\supp(\mu), \supp(\nu))$ such that
\begin{equation} \label{eq: Kantorovich duality formula}
    u_p \left ( \ell_p (\mu, \nu) \right ) = \inf \left \{\int \psi (y) \d \nu (y) - \int \varphi (x) \d \mu (x) \right \} \, ,
\end{equation}
where the infimum is taken over all bounded and continuous $\varphi: E \rightarrow \R$ and $\psi: E \rightarrow \R$ such that $\psi (y) - \varphi (x) \geq u_p(\ell(x,y))$ on $E \times E$. 
In particular, the formula in \eqref{eq: Kantorovich duality formula} holds even if we consider the infimum among $\psi \in L^1(\nu)$, $\varphi \in L^1(\mu)$ such that $\psi (y) - \varphi(x) \geq u_p(\ell(x,y))$ on $E \times E$. Moreover, 
\begin{equation}
\label{eq: kantorovich infimizers}
u_p \left ( \ell_p (\mu, \nu) \right )  = \inf \left \{  \int \varphi^{c_p}_E (y) \d \nu (y)- \int \varphi (x) \d \mu (x) \right \} \, ,
\end{equation}
where the infimum is taken over all continuous and bounded steep functions $\varphi$ on $K$.
\end{itemize}
\end{theorem}

\begin{proof}
The first point i) follows in the same way of \cite[Theorem 4.1]{villanibook} and the first part of ii) follows by \cite[Theorem 5.9]{villanibook} and Lemma \ref{lem: finding emralds}, picking $E$ as in there. We are left to prove the formula  \eqref{eq: kantorovich infimizers}. Notice that if $(\psi,\phi)$ are admissible continuous functions in \eqref{eq: Kantorovich duality formula}, then $(\varphi,\varphi^{c_p}_E)$ is also so (continuity of $\varphi^{c_p}_E$ follows from Lemma \ref{lem: basic properties of c-transform}) and clearly $\varphi^{c_p}_E\leq\psi$. This suffices to prove that in \eqref{eq: Kantorovich duality formula} we can limit ourselves to couples of the form $(\varphi,\varphi^{c_p}_E)$. To conclude it thus only remains to prove that we can limit ourselves to considering steep functions.

Recall from  \cite[Theorem 1.3]{Minguzzisteep}  that by global hyperbolicity there exists a   smooth $1$-steep function $\mathbf t: M \to \R$. Let $\eps>0$ and notice that the function $\varphi + \varepsilon \mathbf t$ is a continuous $\varepsilon$-steep function.
Moreover, by definition of $c_p$-transform we have
\[
( \varphi + \varepsilon \mathbf t )^{c_p}_E(y) - \varphi^{c_p}_E(y) \leq \varepsilon \mathbf t (y) \qquad\forall y\in K
\]
and therefore
\begin{align*}
\int (\varphi + \varepsilon \mathbf t )^{c_p}_E \d \nu- \int \varphi + \varepsilon \mathbf t \d \mu & \leq \int \varphi^{c_p}_E(y)  \d \nu(y) - \int \varphi (x) \d \mu(x) + \varepsilon \int \mathbf t \d [\nu - \mu] \\
& \leq \int \varphi^{c_p}_E(y)  \d \nu(y) - \int \varphi (x) \d \mu(x) + 2\varepsilon \|\mathbf t\|_{C_b(K)}.
 \end{align*}
The conclusion follows by arbitrariness of $\varepsilon >0$ and $\varphi$.
\end{proof}

\section{The causal continuity inequality}

\subsection{Analysis on globally hyperbolic spacetimes}

We start this section with the definition of causal continuity inequality. This definition will be crucial for the rest of the paper.

\begin{definition} \label{def: supersolution in the causal sense}
Let $(M,g)$ be a globally hyperbolic spacetime. We say that $(\mu_t,v_t)$ is a solution of the causal continuity inequality  \eqref{eq: causal continuity inequality}, or satisfies \eqref{eq: causal continuity inequality} in short, provided the following hold:
\begin{itemize}
\item[i)] $(\mu_t)$ is a causal curve defined on $[0,1]$ of probability measures on $M$ and for some $K\subset M$ compact we have $\supp(\mu_t) \subset K$ for every $t\in[0,1]$;
\item[ii)] $v_t:M\to TM$ is a Borel  causal vector field for every $t\in[0,1]$;
\item[iii)] For every smooth and  causal function $\varphi:M\to\R$ we have
    \begin{equation} \label{eq: causal continuity inequality}
    \tag{CCI}
    \dfrac{\d}{\d t} \left ( \int \varphi(x) \, \d \mu_t(x) \right ) - \int \d\varphi_x (v_t(x)) \, \d \mu_t(x) \geq 0 \qquad a.e.\ t\in[0,1].
    \end{equation}
\end{itemize}
\end{definition}
Notice that for a causal curve of (compactly supported) probability measures $(\mu_t)$ and $\varphi: M\to\R$ causal, the map $\Phi:[0,1]\to\R$ defined as 
\begin{equation}
\label{eq:Phi}
\Phi(t):=\int\varphi\,\d\mu_t
\end{equation}
is monotone, thus the derivative in \eqref{eq: causal continuity inequality} exists for a.e.\ $t\in[0,1]$ and the definition is well posed. To see monotonicity, let $s\leq t$, $ \pi \in \Pi_{\leq}(\mu_s, \mu_t)$ and notice that
\[
0 \leq \int \varphi(y) - \varphi(x) \d \pi(x,y) = \int \varphi (y) \d \mu_t (y) - \int \varphi (x) \d \mu_s(x) = \Phi(t) - \Phi(s).
\]
In particular, we see that if $(\mu_t)$ is causal, then $(\mu_t,0)$  always satisfies \eqref{eq: causal continuity inequality}.
\begin{remark}[Distributional formulation of \eqref{eq: causal continuity inequality}] \label{rem: distributional formulation}
Note that \eqref{eq: causal continuity inequality} can equivalently be formulated as a distributional inequality.
Indeed, the couple $(\mu_t, v_t)$ satisfies \eqref{eq: causal continuity inequality} if and only if 
\begin{equation} \label{eq: distributional causal continuity inequality}
D\Phi \geq g(t) \d t \qquad \text{ holds in } \mathcal{D}'(0,1), \qquad\text{where}\qquad  g(t) := \int \d\varphi_x(v_t(x)) \d \mu_t(x),
\end{equation}
for every $\varphi: M \rightarrow \R$ smooth causal function. Notice that that $g$ is well defined since the integrand is non-negative. \\

Indeed, we already noticed that $\Phi$ is monotone, thus its distributional derivative is (representable by) a non-negative measure $D\Phi$ and the Lebesgue Decomposition Theorem allows to write
\[
D\Phi -g(t) \d t = \left [ \Phi'(t) -g(t) \right ] \d t +(D\Phi)^s.
\]
The monotonicity of $\Phi$ ensures that $(D\Phi)^s\geq 0$, from which it easily follows that \eqref{eq: causal continuity inequality} holds if and only if $D\Phi -g(t) \d t\geq 0$, as claimed.
\end{remark}

The following lemma is a technical one and it is needed in the proof of Proposition \ref{prop: part (A) in glob hyper}. 
It is a slight generalization of \cite[Theorem 2.55]{Min19} and for completeness we give the full proof.
\begin{lemma} \label{lem: control on the derivative}
    Let $(M,g)$ be a globally hyperbolic spacetime and let $K$ be a compact set in $M$. Let $h$ be a Riemannian metric on $M$, then there exists a constant $C=C(K,h) \in (0, +\infty)$, such that 
    \begin{equation} \label{eq: estimate with riemannian metric}
        \int_0^1 \sqrt{h_{\gamma_t}(\dot \gamma_t, \dot \gamma_t)} \d t \leq C
    \end{equation}
    for every causal curve $\gamma :[0,1] \rightarrow K$.
\end{lemma}

\begin{proof}
    Let $U_i$ for $i=1 , \dots k$ be a finite open covering of $K$ given by causally convex sets. Pick $\varphi^i_1, \dots , \varphi^i_n$ local coordinates such that $\nabla \varphi^i_j$ are timelike vector fields on $U_i$ for $j=1, \dots n$ and define the Riemannian metric $h^i$ on $U_i$ by declaring that the vector fields are orthonormal and with unit modulus. Let $\gamma: [0,1] \rightarrow U_i$ be a causal curve, then at a.e. $t \in [0,1]$ we have that
    \begin{align*}
        \sqrt{h^i_{\gamma_t} (\dot \gamma_t, \dot \gamma_t)} = \left [ \sum_{j=1}^n \d \varphi^i_j ( \dot \gamma_t )^2 \right ]^{\frac{1}{2}} \leq \sum_{j=1}^n \d\varphi^i_j (\dot \gamma_t),
    \end{align*}
    having used that $\d\varphi^i_j(\dot\gamma_t)\geq 0$ for every $i,j$.     Since any $t \mapsto \varphi^i_j(\gamma_t)$ is non-decreasing we have that
    \begin{align*}
        \int_0^1 \sqrt{h^i_{\gamma_t} (\dot \gamma_t, \dot \gamma_t)} \, \d t \leq \sum_{j=1}^n \varphi^i_j (\gamma_1) - \varphi^i_j(\gamma_0) \leq 2n \max_{j=1 , \dots n} \|\varphi^i_j\|_{L^{\infty}(U_i)} \, .
    \end{align*}
    Let now $\rho_i$ for $i=1, \dots k$ be partitions of unity subordinated to the open covering $U_i$ and define $h_x: = \sum_{i=1}^n \rho_i(x) h^i_x$. It is clear that $h$ is a Riemannian metric on $K$ that can be extended to to the whole $M$. Let $\gamma:[0,1] \rightarrow K$ be a causal curve, it is clear that it can be written as concatenation of at most $k$ causal curves, where each one is supported on a different neighborhood $U_i$. In particular
    \begin{align*}
        \int_0^1 \sqrt{h_{\gamma_t} (\dot \gamma_t, \dot \gamma_t)} \, \d t \leq \sum_{i=1}^k \int_0^1 \sqrt{\rho_i(\gamma_t) \, h^i_{\gamma_t}(\dot \gamma_t, \dot \gamma_t) } \, \d t \leq 2kn \max _{i,j} \|\sqrt{\rho_i} \|_{L^{\infty}(K)} \|\varphi_j\|_{L^{\infty}(U_i)}.
        \end{align*}
    The conclusion follows from the fact that any two Riemannian metrics are locally Lipschitz equivalent.
\end{proof}

Now we state and prove the first result about the causal continuity inequality, namely part (A) of Theorem \ref{Thm: main}. Here and below, given a Borel probability measure $\mu$, a causal vector field $v$ and  $0 \neq p <1$ we define $\|v\|_{L^p_\mu}$ as
\[
\|v\|_{L^p_\mu}:=\Big(\int |v|^p\,\d\mu\Big)^{\frac1p},
\]
being intended that $0^p$ is equal to $0$ if $p\in(0,1)$ and to $+\infty$ if $p<0$. We have:
\begin{proposition}
\label{prop: part (A) in glob hyper}
Let $(M,g)$ be a globally hyperbolic spacetime and let $0 \neq p <1$. 
Let $[0,1] \ni t \mapsto \mu_t \in \Prob(M)$ be a causal path such that $\supp( \mu_t) \subset K$ for every $t \in [0,1]$, where $K$ is a compact subset of $M$. 
Suppose that $\int_0^1u_p\left ( | \dot \mu_t |_p \right )\,\d t>-\infty$.
 
Then there exists a family of causal Borel vector fields $(v_t)_{t \in [0,1]}$ such that: \begin{itemize}
\item[\emph{i)}] $ u_p \left ( \| v_t \|_{L^p_{\mu_t}} \right ) \geq u_p\left ( | \dot \mu_t |_p \right )$ holds for almost every $t \in [0,1]$,
\item[\emph{ii)}] the couple $(\mu_t, v_t)$ satisfies the causal continuity inequality \eqref{eq: causal continuity inequality}.
\end{itemize}
\end{proposition}

\begin{proof}
Notice that if we modify the path $t \mapsto \mu_t$ on a negligible set of times, the family of Borel vector fields satisfying $(i)$ and $(ii)$ remains unchanged. Hence, from now on, we replace the path $t \mapsto \mu_t$ with its left-continuous representative (recall \cite[Proposition 2.28]{Octet}). Using the Lifting Theorem \ref{Th:Lifting}, we can find $\bdpi \in \Prob (\LCC([0,1]; M))$ such that $(\e_t)_{\#} \bdpi = \mu_t$ and
\begin{equation} \label{eq: pointise energy identity}
u_p \left ( | \dot \mu_t |_p \right ) = \int u_p \left ( \| \dot \gamma _t \|_g \right ) \, \d \bdpi (\gamma)
\end{equation}
for almost every $t \in [0,1]$. 
We define the Borel map $D_t : \LCC([0,1]; M) \rightarrow TM$ by
\begin{equation}
\LCC([0,1]; M) \ni \gamma \mapsto D_t(\gamma) =  \begin{cases}
( \gamma_t , \dot \gamma_t) & \text{if $\gamma$ is differentiable at $t$} \, , \\
(\gamma_t, 0) & \text{otherwise} \, . 
\end{cases}    
\end{equation}
%Then $(D_t) _{\#} \bdpi \in \Prob (TM)$. 
Let $\Pr : TM \rightarrow M$ be the canonical projection. 
Then we have that
$\Pr _{\#} \left ( (D_t)_{\#} \bdpi \right ) = \left ( \Pr \circ D_t \right )_{\#} \bdpi = (\e_t)_{\#} \bdpi = \mu_t$. 
We disintegrate $(D_t)_{\#} \bdpi \in \Prob(TM)$ with respect to $\Pr$ 
%in the sense of the Disintegration Theorem \ref{disintegration} 
and get
\begin{equation}
(D_t )_{\#} \bdpi = \int \left ( (D_t)_{\#} \bdpi \right )_x \d \mu_t(x) \, .
\end{equation}
We recall that $\left ( (D_t)_{\#} \bdpi \right )_x$ is concentrated on $T_xM = \Pr^{-1}(\{x\})$ for $\mu_t$-almost every $x \in M$. 
If we now also disintegrate $\bdpi$ with respect to $\Pr \circ D_t$, then, by uniqueness of the disintegration, we obtain $ (D_t)_{\#} \bdpi _x = \left ( (D_t)_{\#} \bdpi \right )_x$ for $\mu_t$-almost every $x \in M$. 
We aim to define $v_t(x)$ as the barycenter of the measure $\left ( (D_t)_{\#} \bdpi \right )_x$ defined on the tangent space $T_xM$, i.e.,  
\begin{equation} \label{eq: barycenter}
v_t(x) : = \int _{T_xM} v \, \d \left ( (D_t)_{\#} \bdpi \right )_x (v)= \int D_t(\gamma) \, \d \bdpi_x (\gamma) \qquad \mu_t-a.e.\ x .
\end{equation}
We need to show that the definition in \eqref{eq: barycenter} is well posed. To see this, it is sufficient to prove that, for some auxiliary Riemannian metric $h$ on $M$ we have that
\[
\int \|v\|_h \, \d \left ( (D_t)_{\#} \bdpi \right )_x (v) = \int \|\dot \gamma_t \|_h \, \d \bdpi_x (\gamma) < \infty\qquad \mu_t-a.e.\ x,
\]
 where $\|v \|_h:= \sqrt{h_x(v,v)}$ for $v \in T_xM$. 
Using Lemma \ref{lem: control on the derivative}, there is $C>0$ such that
\begin{equation} \label{eq: energy estimate}
\int_0^1 \int \int \| \dot \gamma_t \|_h \, \d \bdpi_x (\gamma) \, \d \mu_t(x) \d t \leq C < \infty,
\end{equation}
and so we have that $\int \| \dot \gamma_t \|_h \, \d \bdpi_x (\gamma) < \infty$ for $\mu_t$-almost every $x$ and almost every $t \in [0,1]$. \\
We notice that $v_t(x)$ is a causal vector for $\mu_t$-almost very $x \in M$. 
Indeed, for every causal vector $w \in T_xM$ we have that
\[
g_x(w, v_t(x)) = \int g_x(w , \dot \gamma_t) \, \d \bdpi_x (\gamma) \geq 0 \, ,
\]
having used that  $\bdpi_x$ is concentrated on causal curves. 
It is clear from the construction that, up to setting $v_t=0$ on a set of %$\mu_t$ 
measure zero with respect to $\mu_t$, $v_t$ is a causal Borel vector field such that \eqref{eq: barycenter} holds for almost every $t \in [0,1]$ and $\mu_t$-almost every $x \in M$. 

Notice that $i)$ follows as a consequence of the identity in \eqref{eq: pointise energy identity} and Jensen inequality, indeed
\begin{align*}
u_p \left ( \| v_t \|_{L^p_{\mu_t}} \right ) & = \int u_p \left ( \| v_t (x)\|_g \right ) \, \d \mu_t(x) =  \int u_p \left ( \left \| \int D_t(\gamma) \d \bdpi_x(\gamma) \right \|_g \right ) \d \mu_t(x) \\
& \geq\int \int u_p \left ( \| D_t(\gamma) \|_g \right ) \, \d \bdpi _x (\gamma) \d \mu_t(x) = \int u_p \left ( \| \dot \gamma_t \|_g \right ) \, \d \bdpi (\gamma) = u_p \left ( | \dot \mu_t |_p \right ) \,.
\end{align*}
Now we are left to prove $ii)$, namely that the couple $(\mu_t, v_t)$ satisfies \eqref{eq: causal continuity inequality}. Let $\varphi \in C^{\infty}(M)$ be a causal function and define $\Phi :[0,1] \rightarrow \R$ by $\Phi(t) = \int \varphi \d \mu_t$. 
Notice that $\Phi$ is a bounded as $\| \Phi \|_{L^{\infty}} \leq \sup_K |\varphi |$ and non-decreasing (as discussed after \eqref{eq:Phi}). 
Hence $\Phi$ is differentiable almost everywhere by Lebesgue Differentiation Theorem and for a.e. $t \in [0,1]$ we can apply Fatou Lemma, the standard chain rule and Fubini's Theorem to get
\begin{align*}
\Phi '(t) & = \varliminf_{h \downarrow 0} \frac{\Phi(t+h)-\Phi(t)}{h} = \varliminf_{h \downarrow 0} \int \frac{\varphi(\gamma_{t+h})- \varphi(\gamma_t)}{h} \d \bdpi(\gamma) \\
& \geq \int \varliminf_{h \downarrow 0} \frac{\varphi(\gamma_{t+h})- \varphi(\gamma_t)}{h} \d \bdpi(\gamma) = \int \d\varphi_{\gamma_t} (\dot \gamma_t) \, \d \bdpi(\gamma).
\end{align*}
Thus from
\begin{align*}
\int \d\varphi_{\gamma_t} (\dot \gamma_t) \, \d \bdpi(\gamma) & = \int \d \varphi_x(v) \d (D_t)_{\#}\bdpi(x,v) \\
& = \int \int_{T_xM} \d \varphi_x(v) \d (D_t)_{\#} \bdpi_x(v)\d \mu_t(x) \\
& = \int \d \varphi_x \left( \int_{T_xM}v \d (D_t)_{\#} \bdpi_x(v) \right) \d \mu_t(x)  = \int\d\varphi_x(v_t(x)) \d \mu_t(x) \, 
\end{align*}
the conclusion follows.
\end{proof}

\begin{remark} \label{rem: first proof on part A in Minkowski}
Here we explicitly compute the left hand side of \eqref{eq: causal continuity inequality} in the simplified but meaningful case $M= \R^{1,n}$ expanding the computations of Remark \ref{rem: distributional formulation}. Let $t \mapsto \mu_t \in \Prob (M)$ be a causal path and denote by $v_t:M \rightarrow TM$ the causal Borel vector fields constructed in the proof of Proposition \ref{prop: part (A) in glob hyper}. 

Let $\varphi \in C^{\infty}(M)$ be a causal function and consider the bounded non-decreasing function $\Phi(t) = \int \varphi \d \mu_t$. In particular, the derivative in the sense of distributions $D\Phi$ is a non-negative finite measure on $[0,1]$. Using Lebesgue's Decomposition Theorem we can uniquely decompose 
\begin{equation} \label{eq: RD one}
D \Phi = \Phi' (t)  \d t +  D\Phi^s 
\end{equation} 
into its absolutely continuous and singular parts, both of which are non-negative since $D \Phi$ is non-negative. 
Our next step is to investigate these parts further. 
Using the definition of the distributional derivative and Fubini's Theorem, it is easy to check that $D \Phi = \int D(\varphi \circ \gamma) \, \d \bdpi(\gamma)$. 
Hence by a simple application of Vol'pert Chain Rule Theorem \cite[Theorem 3.96]{AmbrosioFuscoPallara} to $D(\varphi \circ \gamma)$ we have that
\begin{equation} 
\label{eq: 1 part (A) Minkowski}
D \Phi = \int \d\varphi _{\gamma_t} ( \dot \gamma_t ) \, \d \bdpi ( \gamma) \d t+ \int D(\varphi \circ \gamma)^s \, \d \bdpi (\gamma) \, ,
\end{equation}
where $D(\varphi \circ \gamma)^s$ denotes the singular part in the Lebesgue decomposition of $D(\varphi \circ \gamma)$.\\
Since $\int D(\varphi \circ \gamma)^s \, \bdpi (d\gamma)$ is a non-negative finite measure, we can use Lebesgue decomposition once more to write 
\begin{equation} \label{eq: 2 part (A) Minkowski}
\int D(\varphi \circ \gamma)^s \, \d\bdpi (\gamma) = \psi (t) \d t + \nu^s \, ,
\end{equation}
where $\psi$ is the Radon-Nikodym density and $\nu^s$ is the singular part. 
Note that $\psi$ is a non-negative function almost everywhere and $\nu^s$ is a non-negative measure. 
Putting together \eqref{eq: 1 part (A) Minkowski} and \eqref{eq: 2 part (A) Minkowski} we have 
\begin{equation} \label{eq: 2 RD dec}
    D\Phi = \left ( \int \d\varphi_{\gamma_t}(\dot \gamma_t) \, \d \bdpi (\gamma) + \psi (t) \right ) \d t + \nu^s.
\end{equation}
Hence \eqref{eq: RD one} and \eqref{eq: 2 RD dec} are two Lebesgue decompositions of $D\Phi$, and by uniqueness the absolutely continuous parts must agree almost everywhere. 
Thus, for almost every $t \in [0,1]$ we have
\[
\Phi'(t) - \int \d\varphi_{\gamma_t}  (\dot \gamma_t) \, \d \bdpi(\gamma) = \psi(t) \geq 0 \, .
\]
Moreover, out of the identity \eqref{eq: 2 part (A) Minkowski} we obtain that $\psi (t) = \left [ \int D (\varphi \circ \gamma)^s \, \d \bdpi (\gamma) \right ]^{ac} (t)$ for almost every $t \in [0,1]$. 
\hfill$\blacksquare$
\end{remark}

The following simple example highlights what gets `lost' along the way in the proof to end up only with an continuity inequality instead of a continuity equation as in positive signature. 

\begin{example} \label{example}
Let $x \leq y$ be two causally related points in $\R^{1,n}$ and consider the map $[0,1] \ni t \mapsto \mu_t = (1-t) \delta_x + t \delta_y \in \Prob (\R^{1,n})$. 
Notice that this curve is continuous with respect to the narrow topology on $\Prob (\R^{1,n})$. 
Moreover, it is causal and 
\[
u_p (\ell_p(\mu_s, \mu_t)) =    \begin{cases}
    \frac{1}{p} \|y-x\|^p (t-s) & p \in (0,1) \\
    - \infty & p <0 \, .
\end{cases}
\]
It is immediate to check the lifting plan $\bdpi \in \Prob(\LCC([0,1], \R^{1,n})$ constructed in Theorem \ref{Th:Lifting} is given by
\[
\bdpi =G _{\#} (\mathcal{L}^1 \mres [0,1]) \, , 
\]
where $G: [0,1] \rightarrow     \LCC([0,1], \R^{1,n})$ is given by 
\[
G(s) = \gamma^{s} , \qquad \gamma^{s}_t = \gamma^{s}(t) = \begin{cases}
  x & 0 \leq t \leq s \\
  y & s < t \leq 1 \, .
\end{cases}
\]
We observe that, given a smooth causal function $\varphi \in C^{\infty}( \R^{1,n})$, it is clear that $\Phi(t) = \int \varphi \d \mu_t = (1-t) \varphi (x) +t \varphi(y)$ and so 
\[
\Phi'(t) = \varphi(y)- \varphi(x) \, .
\]
Moreover we can compute the family of Borel vector fields constructed in Proposition \ref{prop: part (A) in glob hyper} and we end up with $v_t(x) =0$ for $\mu_t$-almost everywhere. \hfill$\blacksquare$
\end{example}

\subsection{A Hopf--Lax type formula}

In this subsection we will develop the machinery to show the part (B) of Theorem \ref{Thm: main}. Starting from now, we will not only work with $0 \neq p <1$, but also with its Hölder conjugate, which we will always denote by $q$. 
That is, $0 \neq q <1$ and $p^{-1} + q^{-1} = 1$. 

\begin{definition}[Lorentzian Hopf--Lax semigroup]
Let $(M,g)$ be a globally hyperbolic spacetime and let $0 \neq p <1$. 
Let $E \subset M$ be a compact subset and let $f: E \rightarrow \R$ be a steep function. 
For $t \in (0,1]$, we define $Q_t^p f:E\to\overline\R$ as
\begin{equation} 
\label{eq: Hopf--Lax Semigroup}
Q_t^p f(y) := \sup_{J^-(y)\cap E} \left \{ f(x) +t \,  u_p \left( \dfrac{\ell(x,y)}{t} \right) \right \}.
\end{equation}
\end{definition}

Next we show some structural properties of the Lorentzian Hopf--Lax semigroup in analogy to the Riemannian setting, see for instance \cite[Theorem 16.10]{AmbrosioBrueSemola}. 
\begin{lemma}[Properties of $Q_t^p$]
\label{lem: properties of Hopf--Lax}
Let $(M,g)$ be a globally hyperbolic spacetime and let $ 0 \neq p <1$. 
Let $K \subset M$ be a compact subset and let $f: E \rightarrow \R$ be a continuous $L$-steep function for some $L > 0$, where $E$ is a compact neighborhood of $K$ as in Lemma \ref{lem: finding emralds}. 
Then: 
\begin{enumerate}
\item[(o)] For any $y\in E$ and $t>0$ the supremum in \eqref{eq: Hopf--Lax Semigroup} is a maximum.
\item[(i)] For any $t \in (0,1]$ we have that $Q_t^p f : K \rightarrow \R$ is continuous and causal. Moreover, for any $y \in K$ we have that
\begin{equation} \label{eq: Hopf Lax on a smaller set}
Q_t^p f(y) = \sup \left \{f(x) +t \,  u_p \left( \dfrac{\ell(x,y)}{t} \right ) : x \in E\, , \, 0 < \ell(x,y) \leq t L^{\frac{1}{p-1}} \right \} \, .
\end{equation}
\item[(ii)] $Q_t^p f : K \rightarrow \R$ is an $L$-steep function for any $t \in (0,1]$.
\item[(iii)] For every $y \in E$, the map $(0,1] \ni t \mapsto Q_t^K f(y)$ is non-decreasing for $p \in (0,1)$ and non-increasing for $p <0$. 
Moreover, for every $\varepsilon>0$, the map $[\varepsilon,1] \ni t \mapsto Q_t^pf(y)$ is Lipschitz continuous uniformly in $y \in K$, with Lipschitz constant bounded by $L,p,\eps$ only.
\item[(iv)] For every $y \in K$, we have that $Q_t^pf(y) \downarrow f(y)= Q_0^pf(y)$ for $t \downarrow 0$ if $p \in (0,1)$, and $Q_t^pf(y) \uparrow f(y) = Q_0^pf(y)$ for $t \rightarrow 0^+$ if $p <0$. 
In particular $Q_t f$, converges to $f$ uniformly on $K$.
\end{enumerate}

\end{lemma}

\begin{proof} \ 

(o) Direct consequence of the continuity of $f$ on $E$ and of $u_p(\ell(\cdot,y)): J^-(y)\cap E\to\R\cup\{-\infty\}$.

(i) Let $y \in K$ and $t \in (0,1]$. 
First we show that $Q_t^p f$ is bounded. 
Indeed, for $p \in (0,1)$ we can test \eqref{eq: Hopf--Lax Semigroup} with $x=y$ in order to get $f(y) \leq Q_t^p f(y)$. 
If $p <0$, since $E$ is a compact neighborhood of $K$ there exists $E \ni \overline{y} \ll y$ and so
\[
- \infty < f( \overline{y}) + t \,u_p \left ( \tfrac{\ell(\overline{y},y)}{t} \right ) \leq Q_t^p f(y) \, .
\]
An upper bound can easily be established noticing that $\ell$, and thus $u_p\circ\ell$, is bounded from above on $E\times E$.

The continuity and causality of $Q_t^pf$ on $K$ follows in the same way of Lemma \ref{lem: basic properties of c-transform}. Next we will prove \eqref{eq: Hopf Lax on a smaller set}. Notice that we can restrict the domain in \eqref{eq: Hopf--Lax Semigroup} to $I^-(y) \cap E$ because of the choice of $E$, Lemma \ref{lem: finding emralds} and the continuity of the function to be maximized. To prove \eqref{eq: Hopf Lax on a smaller set} it thus suffices to show that for  $y \in K$ and $x \in J^-(y) \cap E$ maximizer of \eqref{eq: Hopf--Lax Semigroup} we have $\ell(x,y)\leq tL^{\frac1{p-1}}$. For such $y,x$, we can assume that $\ell(x,y)>0$ (or else there is nothing to prove) and then let $\gamma:[0,1]\to M$ be a maximizing geodesic from $x$ to $y$. Notice that $E$ is an emerald by construction, hence causally convex, and so $\gamma_{\lambda} \in E$ for all $\lambda \in [0,1]$. Thus we have
\begin{equation} \label{eq: p negative restriction hopf lax}
f(x) + t \, u_p \left ( \tfrac{\ell(x,y)}{t} \right ) \geq f(\gamma_{\lambda}) + t \, u_p \left ( \tfrac{\ell( \gamma_{\lambda}, y)}{t} \right )\qquad\forall \lambda\in[0,1]
\end{equation}
and rearranging
\begin{align*}
\frac{1-(1-\lambda)^p}{p t^{p-1}} \, \ell(x,y)^p & =  
t \, u_p \left ( \tfrac{\ell(x,y)}{t} \right ) - t \, u_p \left ( \tfrac{\ell( \gamma_{\lambda}, y)}{t} \right )  \geq f(\gamma_{\lambda}) - f(x) \geq L \ell (x, \gamma_{\lambda}) = L \lambda \, \ell (x,y) \, .
\end{align*}
Dividing by $\lambda$ and letting $\lambda\downarrow0$ we get the $\ell(x,y)\leq tL^{\frac1{p-1}}$, as desired.

(ii) Let $y_1 \leq y_2$ in $K$ and pick $x \in J^-(y_1) \cap E$ to be a maximizer of \eqref{eq: Hopf--Lax Semigroup}. By $(i)$ we know that $\ell(x,y_1) \leq t L^{\frac{1}{p-1}}$. Then, using the reverse triangle inequality and the monotonicity of $u_p$ we compute
\begin{align*}
    Q_t^p f(y_2) - Q_t^p f(y_1) & \geq f(x) + t u_p \left ( \frac{\ell(x,y_2)}{t} \right ) - f(x) - t u_p \left ( \frac{\ell(x,y_1)}{t} \right ) \\
    & \geq t u_p \left ( \frac{\ell(x,y_1) + \ell(y_1,y_2)}{t} \right ) - t u_p \left ( \frac{\ell(x,y_1)}{t} \right ) \, .
\end{align*}
In particular, we have that
\begin{equation}
\label{eq:ders}
    \varliminf_{y_2 \downarrow y_1} \frac{Q_t^p f(y_2) - Q_t^p f(y_1)}{\ell(y_1, y_2)}  \geq u_p' \left ( \frac{\ell(x,y_1)}{t} \right ) \geq L \, .
\end{equation}
Now we are able to prove that $Q_t^p f$ is an $L$-steep function on $K$. Consider $y_1 \leq y_2$ in $K$ and connect them with a maximizing causal geodesic $\gamma_{\lambda}$ such that $\gamma_0= y_1$ and $\gamma_1 = y_2$. Then the function $s(\lambda) : = Q_t^p f(\gamma_{\lambda})$ is non-decreasing and so 
\[
Q_t^p f(y_2) - Q_t^p f(y_1) = s(1) - s(0) \geq \int_0^1 s'(\xi) \d \xi \stackrel{\eqref{eq:ders}}\geq L\ell(y_1, y_2).
\]

(iii) The non-increasing and non-decreasing monotonicity easily follows from the definition in \eqref{eq: Hopf--Lax Semigroup}. 
We are left to prove the Lipschitz continuity. 
We will prove it for $p \in (0,1)$ since the case $p <0$ is completely analogous after replacing $s$ with $t$ in the computation in \eqref{eq: lipschiz cont proof}. 
Let $y \in K$ be fixed and let $\varepsilon >0$. 
Consider $\varepsilon \leq s\leq t \leq 1$, then, we can find a maximizer $x \in J^-(y) \cap E$ (with $x \ll y$ if $p<0$) for $Q_t^pf$ and thus that
\begin{align} \label{eq: lipschiz cont proof}
Q_t^p f(y) - Q_s^p f(y) & \leq f(x) + t \, u_p \left ( \tfrac{\ell(x,y)}{t} \right ) - f(x) - s \, u_p \left ( \tfrac{\ell(x,y)}{s} \right ) \\
& \leq u_p ( \ell(x,y)) \left [ \frac{1}{t^{p-1}} - \frac{1}{s^{p-1}} \right ] \leq u_p \left ( L^{\frac{1}{p-1}} \right ) \left [ \frac{1}{t^{p-1}} - \frac{1}{s^{p-1}} \right ] \, \notag.
\end{align}
Notice that if we swap the role of $s$ and $t$, the estimate is analogous. 
The conclusion follows by the fact that $t \mapsto t^{1-p}$ is Lipschitz on $[\varepsilon, 1]$. 

(iv) Now we want to prove that for some fixed $y \in K$ the map $t\mapsto Q_t^pf(y)$ monotonically converges to $f(y)$ when $t \downarrow 0$. First of all, by monotonicity we have that such a limit exists. Let $p <1$ and fix some $y \in K$. For every $t \in (0,1]$ let $y_t \in J^-(y) \cap E$ be a maximizer  in the definition of $Q_t^pf(y)$. 
Note that by the $L$-steepness of $f$ we have that $f(y_t) \leq f(y) - L \ell(y_t,y)$ thus, using also Young's inequality, we have 
\[
Q_t^p f(y) \leq f(y) - L \ell(y_t,y) +t \, u_p \left (\tfrac{\ell(y_t, y)}{t} \right ) \leq f(y) - t\,u_q \left ( L    \right)\, .
\]
Then just let $t \to 0^+$ on both sides to get that $\lims_{t\downarrow 0}Q_t^p f(y) \leq f(y)$. We are left to prove the reverse inequality. If $p \in (0,1)$ it is trivial because of $f(y) \leq Q_t^pf(y)$. If $p<0$, consider $z \ll y$ and notice that
\begin{align*}
f(z) + t \, u_p \left ( \tfrac{\ell(z,y)}{t} \right ) \leq Q_t^p f(y)\qquad\forall t>0.
\end{align*}
Let  $t \rightarrow 0^+$ and notice that the left hand side of the previous equation converges to $f(z)$ to conclude that $f(z)\leq\limi_tQ_t^p f(y)$. Since $f$ is continuous, recalling also Lemma \ref{lem: finding emralds}  we have $\sup_{z\in I^-(y)\cap E}f=f(y)$. The conclusion follows.
\end{proof}
Under the same assumption of this last lemma, for $y\in K$ and $t>0$ we define $\Lmax(y,t)\geq 0$ as
\begin{equation}
\label{eq:lmin}
\Lmax(y,t):=\sup \big\{\ell(x,y) \ :\ x\text{ is a maximizer in the definition of }Q_t^pf(y)\big\}.
\end{equation}
By compactness it is clear that the supremum is in fact a maximum. The following is also easily established:
el\begin{lemma}
\label{lem: dmax is lsc}
With the same assumptions and notation of Lemma \emph{\ref{lem: properties of Hopf--Lax}}, the following holds.

For any $t>0$ the map $\Lmax(\cdot,t)$  is upper semicontinuous on $K$.
\end{lemma}

\begin{proof}
    Let $(y_n)\subset K$ be converging to some limit point $y\in K$ and for every $n\in\N$ let $z_n\in J^-(y_n)\cap E$ maximizer for the definition of $Q^p_tf(y_n)$ such that $\ell(z_n,y_n)=\Lmax(y_n,t)$. By compactness of $E$, up to pass to a non-relabeled subsequence we can assume that $(z_n)$ converges to a limit point $z\in E$. Then the continuity of $Q_t^pf$ gives
\[
Q_t^pf(y)=\lim_nQ_t^pf(y_n)=\lim_nf(z_n)+u_p(\ell(z_n,y_n))=f(y)+u_p(\ell(z,y)),
\] 
proving that $z$ is a maximizer for the definition of $Q_t^pf$. The conclusion follows.
\end{proof}

Now we prove a Proposition that links the Lorentzian Hopf--Lax semigroup and a Hamilton--Jacobi type equation. 
In the smooth category, this result is a classic. In the metric one, the analogous statement in positive signature has been obtained in \cite{AmbrosioGigliSavareInventiones} and \cite{NCS14}.
\begin{proposition}
\label{prop: differential inequality of hopf lax with local steep} With  the same assumptions and notation of Lemma \emph{\ref{lem: properties of Hopf--Lax}}, the following holds.

For any $y \in K$ the map $t\mapsto Q_t^p f(y)$ is differentiable almost everywhere and 
\begin{equation} \label{eq: Hamilton jacobi}
\dfrac{\d}{\d t} Q_t^pf(y) + u_q \left ( \sta (Q_t^pf) (y) \right ) \geq 0 
\end{equation}
holds for almost every $t \in [0,1]$, where $0 \neq q<1$ is the Hölder conjugate of $p<1$.
\end{proposition}
\begin{proof}
By  the monotonicity property shown in Lemma \ref{lem: properties of Hopf--Lax} we have that the map $t \mapsto Q_t^pf(y)$ is differentiable almost everywhere. 
Let $t \in (0,1)$ be a differentiability point, we claim that
\begin{equation}
\label{eq: estimating Q derivative}
\dfrac{\d}{\d t} Q_t ^pf(y) \geq - \frac{1}{q} \left(\frac{\Lmax (y,t)}{t}\right)^p .   
\end{equation}
We are going to prove this for $p\in(0,1)$ only, the case $p<0$ being analogous (the proof just uses $h<0$ rather than $h>0$ in below). Let $h>0$ be such that $t + h \in [0,1]$ and $x\in J^-(y)\cap E$ be a maximizer in the definition of $Q_t^pf(y)$. Since $x$ is also a competitor in the definition of $Q_{t+h}^pf(y)$ we have
\begin{align*}
Q_{t+h}^pf(y) - Q_t^pf(y) & \geq   (t+h) \, u_p \left ( \tfrac{\ell(x,y)}{t+h} \right ) - t \, u_p \left (\tfrac{\ell(x,y)}{t} \right ) = \tfrac{1}{p} \ell(x,y)^p \left ( \frac{1}{(t+h)^{p-1}} - \frac{1}{t^{p-1}}\right ).
\end{align*}
Further choosing $x$ so that $\ell(x,y)=\Lmax(y,t)$ the claim \eqref{eq: estimating Q derivative} easily follows.

To get the conclusion it will therefore be sufficient to show that
\begin{equation}
\label{eq:pezzo2}
\sta (Q_t^pf) (y))  \geq \left(\frac{\mathsf{L_{max}}(y,t)}{t}\right)^{p-1} \, . 
\end{equation}
To see this, let $U\subset E$ be an open neighborhood of $y$. Then let $z_1,z_2\in U$ be with  $z_1 \ll z_2$ and   $w  \in J^{-}(z_1) \cap E$ be a maximizer for the definition of $Q_t^pf(z_1)$ such that $\ell(w,z_1)=\Lmax(z_1,t)$. The monotonicity and concavity of $u_p$ together with the reverse triangle inequality give
\begin{align*}
Q_t^pf (z_2) - Q_t^pf(z_1) & \geq t \, u_p   ( \tfrac{\ell(w,z_2)}{t}   ) - t \, u_p   ( \tfrac{\ell(w,z_1)}{t}   )
\\
& \geq t\left ( u_p   (\tfrac{ \ell(w,z_1) + \ell(z_1,z_2)}t   ) - u_p ( \tfrac{\ell(w,z_1)}t )\right ) \geq \ell(z_1,z_2) \,u_p'  \big(\tfrac{\ell(z_1,z_2) + \ell(w,z_1)}t \big) .
\end{align*}
It follows that
\[
{\sf st}(Q_t^pf\big{|}_{U} ) \geq \frac{1}{t^{p-1}} \big({\rm diam}(U)+\sup_{z\in U} \Lmax(z,t)\big)^{p-1}, 
\]
where ${\rm diam}(U)$ is the timelike diameter of $U$. As $U$ shrinks to $y$, such diameter goes to 0, so the conclusion \eqref{eq:pezzo2} follows from the upper semicontinuity of $\Lmax$ established in Lemma \ref{lem: dmax is lsc}. 
\end{proof}

\subsection{The Kuwada Duality}

Now we are going to prove some results using the so-called Kuwada Duality, which consists of interpolating between a function and its $c$-transform using the Hopf--Lax semigroup.
We begin with a simple lemma that allows for approximation of causal functions by smooth causal functions on a globally hyperbolic spacetime based on some ideas of \cite{MinCruGra}. 
\begin{lemma}
\label{smothing a causal function}
Let $(M,g)$ be a globally hyperbolic spacetime and $\varphi : M \rightarrow \R$ be a bounded, continuous and causal function. Let $K \subset M$ be a compact set. Then, there exists a family $\varphi_{\varepsilon} : M \rightarrow \R$ of smooth causal functions such that $\varphi_{\varepsilon}\rightarrow \varphi$ pointwise as $\varepsilon \rightarrow 0$ and for every $x \in K$ we have
\[
\varliminf_{\varepsilon \downarrow 0} \| (d \varphi_{\varepsilon})_x \|_{g^{\ast}} \geq \sta (\varphi) (x) \, .
\]
\end{lemma}
\begin{proof}
If $M$ is the Minkowski spacetime, the conclusion follows by standard mollification. In the general case, an argument by charts, akin e.g.\ to that used in \cite[Theorem A.2]{Octet}, does the job. We omit the details.
\end{proof}

We apply the previous lemma to obtain some useful preliminary inequality that we will use in the proof of Theorem \ref{part B}.

\begin{lemma}
\label{lem: ineq 1}
Let $(M,g)$ be a globally hyperbolic spacetime and $(\mu_t,v_t)$ be solving the causal continuity inequality. Then
\[
\int \varphi(x) \, \d \mu_t(x) - \int \varphi(x) \, \d \mu_s(x) \geq \int_s^t \int \sta  (\varphi)(x)\,  \| v_r(x) \|_g \, \d \mu_r(x) \, \d r
\] 
holds for every $\varphi : M \rightarrow \R$ bounded, continuous and causal function and $t \geq s$ in $[0,1]$. 
\end{lemma}

\begin{proof}
First of all we recall that $\Phi(s) = \int \varphi(x) \, \d \mu_s(x)$ is non-decreasing since $\mu_t$ is a causal path and $\varphi$ is a causal function. Suppose for the moment that $\varphi$ is smooth in addition to being bounded and causal. 
Notice that in this case $\varphi$ is a valid test function for the causal continuity inequality yielding $\Phi'(r) \geq \int d \varphi _x (v_r(x)) \, \d \mu_r(x)$. 
Then we compute
\begin{align*}
 \int \varphi(x) \, \d \mu_t(x) - \int \varphi(x) \, \d \mu_s(x)&=\Phi(t)-\Phi(s)  \\ 
& \geq \int_s^t  \Phi'(r) \, \d r \geq \int_s^t \int \d \varphi _x (v_r(x)) \, \d \mu_r(x) \, \d r \\ 
& \geq \int_s^t \int \| \d \varphi _x \|_{g^{\ast}} \, \|v_r(x) \|_g \, \d \mu_r(x) \, \d r \, ,
\end{align*}
where the last inequality follows by the immediate implication of the definition of dual hyperbolic norm \eqref{eq: dual hyperbolic norm}. 

Let now $\varphi : M \rightarrow \R$ be merely bounded, continuous and causal. 
Consider a family $\varphi_{\varepsilon}$ as in Lemma \ref{smothing a causal function}. 
In particular, $\varphi_{\varepsilon}$ is a valid test function for the causal continuity inequality and hence we have that
\[
\int \varphi_{\varepsilon}(x) \, \d \mu_t(x) - \int \varphi_{\varepsilon} (x) \, \d \mu_s(x) \geq \int_s^t \int \| (\d \varphi_{\varepsilon}) _x \|_{g^{\ast}} \,  \| v_r(x) \|_g \, \d \mu_r(x) \, \d r \, .
\]
Now we consider the $\varliminf$ in $\varepsilon$ as $\varepsilon \to 0$ of both expressions. 
On the left hand side we can apply the Dominated Convergence Theorem and on the right hand side we we can apply Fatou's Lemma to get
\begin{align*}
\int \varphi(x) \, \d \mu_t(x) - \int \varphi (x) \, \d \mu_s(x) 
& = \varliminf_{\varepsilon \downarrow 0} \int \varphi_{\varepsilon}(x) \, \d \mu_t(x) - \int \varphi_{\varepsilon} (x) \, \d \mu_s(x) \\ 
& \geq \varliminf_{\varepsilon \downarrow 0} \int_s^t \int \| (\d \varphi_{\varepsilon}) _x \|_{g^{\ast}} \,  \| v_r(x) \|_g \, \d \mu_r(x) \, \d r \\
& \geq \int_s^t \int \varliminf_{\varepsilon \downarrow 0} \| (\d \varphi_{\varepsilon}) _x \|_{g^{\ast}} \,  \| v_r(x) \|_g \, \d \mu_r(x) \, \d r \, .
\end{align*}
Hence by Lemma \ref{smothing a causal function} the conclusion follows.
\end{proof}

Now we state the following result which is inspired from \cite[Lemma 17.7]{AmbrosioBrueSemola} and \cite[Theorem 15.6]{maggiOT}.

\begin{lemma}
\label{lem: ineq 2}
Let $(M,g)$ be a globally hyperbolic spacetime and $(\mu_t,v_t)$ be solving the causal continuity inequality. 

Then there exists a negligible set $N \subset [0,1]$ such that for every causal function $\varphi: M \rightarrow \R$ and $t \not \in N$ we have
\[
\varliminf_{ h \rightarrow 0} \frac{1}{h} \left ( \int \varphi(x) \d \mu_{t+h}(x) - \int \varphi (x) \d \mu_{t}(x) \right ) \geq \int \sta ( \varphi )(x) \|v_t(x)\|_g \d \mu_t(x) \, .
\]
\end{lemma}

\begin{proof}
We claim that there exists a negligible set $N \subset [0,1]$ such that for every lower semicontinuous function $u:M \rightarrow [0,\infty)$ and $t \not \in N$ we have that
\begin{equation} \label{eq: interm ineq}
    \varliminf_{ h \rightarrow 0} \frac{1}{h} \int_{t}^{t+h} \int u(x) \|v_r(x)\|_g \, \d \mu_r(x) \, \d r \geq \int u(x) \|v_t(x)\|_g \d \mu_t(x) \, .
\end{equation}
The conclusion will follow from the fact that $\sta (\varphi): M \rightarrow [0,\infty)$ is lower semicontinuous for any causal function $\varphi$ and Lemma \ref{lem: ineq 1}. \\
At first, observe that it is sufficient to prove \eqref{eq: interm ineq} for continuous functions. 
Indeed, for any lower semicontinuous function $u$ there exist a non-decreasing sequence of continuous functions $u_n : M \rightarrow [0,\infty)$ such that $u(x) = \sup_n u_n(x)$ for any $x \in M$. 
Assuming \eqref{eq: interm ineq} holds true for $u$ continuous, for every $t \not \in N$ we have that
\begin{align*}
\varliminf_{h \rightarrow 0} \frac{1}{h} \int_{t}^{t+h} \int u (x) \|v_r(x)\|_g \, \d \mu_r(x) \, \d r & \geq \varliminf_{h \rightarrow 0} \frac{1}{h} \int_{t}^{t+h} \int u_n (x) \|v_r(x)\|_g \, \d \mu_r(x) \, \d r \\ & =  \int u_n (x) \|v_t(x)\|_g \d \mu_t(x) \, .
\end{align*}
Now we consider the $\varliminf$ as $n\to\infty$ on both sides and obtain the conclusion with the Monotone Convergence Theorem. 
So suppose that $u: M \rightarrow [0, \infty)$ is a continuous function. Moreover, we can assume that the map 
\[
[0,1] \ni t \mapsto \int u(x) \|v_t(x) \|_g \, \d \mu_t(x)
\]
is $L^1_{loc}(0,1)$, otherwise the result is readily verified.
Thus, we can apply the Lebesgue Differentiation Theorem and so there exists $N(u) \subset [0,1]$ negligible such that for every $t \not \in N(u)$ we have that
\[
\lim_{h \rightarrow 0} \frac{1}{h} \int_{t}^{t+h} \int u (x) \|v_r(x)\|_g \, \d \mu_r(x) \, \d r = \int u (x) \|v_t(x)\|_g \d \mu_t(x) \, .
\]
Let $\mathcal{D} \subset C^0(K; \R)$ (here $K\subset M$ is a compact set containing the supports of all the $\mu_t$'s) be a countable and dense set with respect to uniform topology and define
\[
N:= \bigcup_{ \psi \in \mathcal{D} } N( \psi ) \, .
\]
In particular, $N \subset [0,1]$ is negligible since it is a countable union of negligible sets. 
Let $u : M \rightarrow [0, \infty)$ be continuous and  notice that we can find $(\psi_n)\subset \mathcal{D}$ such that $\psi_n \uparrow u$ on $K$. 
Let $t \not \in N$, then we can compute
\begin{align*}
\varliminf_{h \rightarrow 0} \frac{1}{h} \int_{t}^{t+h} \int u(x) \|v_r (x) \|_g \, \d \mu_r(x) \, \d r &\geq  \varliminf_{h \rightarrow 0} \frac{1}{h} \int_{t}^{t+h} \int \psi_n(x) \|v_r (x) \|_g \, \d \mu_r(x) \, \d r \\ & \geq \int \psi_n(x) \|v_t (x) \|_g \, \d \mu_t(x) \, ,
\end{align*}
and the conclusion follows once more by Monotone Convergence Theorem. 
\end{proof}

Now we state an analogous result of \cite[Lemma 4.3.4]{AmbrGigSav} for monotone functions instead of absolutely continuous functions, moreover we deduce an integral estimate out of it.

\begin{lemma}
\label{lem: analogo AmbrGigSav}
Let $\eta: [0,1] \times [0,1] \rightarrow \R$ be a function such that $\eta (\cdot, s)$ is continuous for all $s \in[0,1]$ and, for every $\varepsilon >0$, uniformly Lipschitz on $[\varepsilon,1]$. Assume also that   $\eta(t,\cdot)$ is non-decreasing for every $t \in [0,1]$. We define
\[
r(t) := \varliminf_{h \rightarrow 0^+} \frac{\eta (t,t) - \eta (t-h,t)}{h} + \varliminf_{h \rightarrow 0^+} \frac{\eta(t,t+h) - \eta (t,t)}{h} \, ,
\]
and  define $\delta (t) = \eta(t,t)$. Then $\delta \in L^1$ and $D\delta \geq r(t) \d t$ holds in the sense of distributions. Moreover, we have 
\begin{equation}
\label{eq:delta01}
\delta (1) - \delta (0) \geq \int_0^1 r(\xi) \d \xi \, .
\end{equation}
\end{lemma}

\begin{proof}
It is clear that $\delta$ is bounded and Borel, hence it is in $L^1$. Let $\varphi \in C^{\infty}_c(0,1)$ be a non-negative function and $h>0$ such that $\pm h + \supp(\varphi) \subset (0,1)$. We then obtain 
\begin{equation*}
    - \int_0^1 \delta (t) \frac{\varphi(t+h)- \varphi(t)}{h} \d t = \int_0^1 \varphi(t) \frac{\eta(t,t)- \eta(t-h,t)}{h} \d t + \int_0^1 \varphi(t+h) \frac{\eta(t,t+h)- \eta(t,t)}{h} \d t,
\end{equation*}
by adding and subtracting $\eta(t-h,t)$ inside the integral and making a change of variables. We take on both sides of the previous equation $\varliminf$ for $h \downarrow 0$. On the first and second term we can apply Dominated Convergence Theorem (for the second, the uniformly Lipschitz assumption matters) and on the third term we apply Fatou Lemma. This proves the desired distributional bound. 

We pass to \eqref{eq:delta01}. Using appropriate test functions, for every $\varepsilon>0$, we obtain
\begin{equation} \label{eq: varepsilon thesis}
    \frac{1}{\varepsilon} \left [ \int_{1-\varepsilon}^1 \delta(t) \d t - \int_0^{\varepsilon} \delta(t) \d t \right ] \geq \int_{\varepsilon}^{1-\varepsilon} r(t) \d t\, .
\end{equation}
Given  $\varepsilon>0$ we estimate
\begin{align*}
    \frac{1}{\varepsilon} \int_0^{\varepsilon} \delta (t) \d t \geq \frac{1}{\varepsilon} \int_0^{\varepsilon} \eta(t,0) \, \d t = \int_0^1 \eta (\varepsilon T, 0) \, \d T, 
\end{align*}
hence, by Dominated Convergence Theorem we get $\varliminf_{\varepsilon \downarrow 0} \frac{1}{\varepsilon} \int_0^{\varepsilon} \delta(t) \d t \geq \eta(0,0)= \delta (0).$
Analogously we have that $\varlimsup_{\varepsilon \downarrow 0} \frac{1}{\varepsilon} \int_{1-\varepsilon}^1 \delta (t) \d t \leq \eta (1,1)= \delta (1)$. The conclusion follows by taking the $\varlimsup$ on \eqref{eq: varepsilon thesis} and using Dominated Convergence Theorem.
\end{proof}

Now we apply Kuwada Duality in order to prove part $(B)$ of Theorem \ref{Thm: main}. 
\begin{proposition}
\label{part B}
Let $(M,g)$ be a globally hyperbolic spacetime and $(\mu_t,v_t)$ be solving the causal continuity inequality. Let $0 \neq p <1$.

Then
\begin{equation} \label{eq: step 2 claim}
\int_0^1u_p(| \dot \mu_t |_p) \,\d t\geq\int_0^1 u_p \left ( \| v_t \|_{L^p_{\mu_t}} \right ) \,\d t   \, .
\end{equation}
Moreover, if $\int_0^1 u_p(\|v_t\|_{L^p_{\mu_t}}) \, \d t >-\infty$ then we have
\begin{equation} \label{eq: step 2 pointwise}
u_p(| \dot \mu_t |_p) \geq u_p \left ( \| v_t \|_{L^p_{\mu_t}} \right ), \quad \text{ for almost every } t \in [0,1] \, .
\end{equation}
\end{proposition}
\begin{proof} We shall prove that
\begin{equation}
\label{eq:claimperB}
u_p(\ell_p(\mu_0,\mu_1))\geq \int_0^1 u_p \left ( \| v_t \|_{L^p_{\mu_t}} \right ) \,\d t   \, .
\end{equation}
After a simple rescaling argument and recalling the formula
\[
\int_0^1u_p(| \dot \mu_t |_p)\,\d t=\sup\big\{\sum_{i=0}^{n-1}(t_{i+1}-t_i)u_p(\tfrac{\ell_p(\mu_{t_i},\mu_{t_{i+1}})}{t_{i+1}-t_i})\ :\ n\in\N , \, 0\leq t_0\leq\cdots t_n\leq1\big\}
\]
obtained in \cite[Theorem 2.23]{Octet}, the claim \eqref{eq: step 2 claim} follows. Then \eqref{eq: step 2 pointwise} is a direct consequence of the additional integrability assumption and existence of Lebesgue points. We thus focus on \eqref{eq:claimperB}.

Let $K\subset M$ be compact containing the supports of all the $\mu_t$'s (this exists by our definition of `solving the causal continuity inequality') and then let $E\subset M$ be a compact neighborhood of $K$ as given by Theorem \ref{thm: kantorovich duality}. Let $\varphi:E\to \R$ be a steep and continuous function, put $Q^p_0 \varphi := \varphi$ and notice that   $Q^p_1 \varphi = \varphi^{c_p}_E$ on $K$. Then   define the function $\eta: [0,1] \times [0,1] \to \R$ as 
\[
\eta (t,s) = \int Q^p_t \varphi (x) \, \d \mu_s (x) \, .
\]
Note that $\eta(\cdot , s)$ is continuous and uniformly Lipschitz on $[\varepsilon', 1]$ for any $\varepsilon' >0$ because of points $(iii)$ and $(iv)$ of Lemma \ref{lem: properties of Hopf--Lax}.  Moreover $\eta(t, \cdot)$ is non-decreasing since $Q_t^p \varphi$ is causal. 

The function $r$ defined in  Lemma \ref{lem: analogo AmbrGigSav} is then given by
\begin{align*}
r(t) = \varliminf_{h \rightarrow 0^+} \frac{1}{h} \int Q_{t}^p \varphi (y) - Q_{t-h}^p \varphi (y) \d \mu_t(y) + \varliminf_{h \rightarrow 0^+} \frac{1}{h} \int Q_t^p \varphi (y) \d [ \mu_{t+h} - \mu_{t}](y) \, .
\end{align*}
Using Lemma \ref{lem: ineq 2} and the reverse Young inequality we compute
\begin{align*}
\varliminf_{h \rightarrow 0^+} \frac{1}{h} \int Q_t^p \varphi (y)\d [ \mu_{t+h} - \mu_{t}](y) & \geq \int \sta ( Q_t^p \varphi ) \, \| v_t \|_g \, \d\mu_t \\
& \geq \int u_q\left ( \sta ( Q_t^p \varphi ) \right ) + u_p( \|v_t \|_g) \, \d \mu_t \qquad a.e.\ t .
\end{align*}
Since   $[\varepsilon',1] \ni t \mapsto Q_t^p \varphi (y)$ is uniformly Lipschitz, the Dominated Convergence theorem together with Fubini's Theorem yields 
\[
\varliminf_{h \rightarrow 0^+} \frac{1}{h} \int Q_{t}^p \varphi (y) - Q_{t-h}^p \varphi (y) \d \mu_t(y) = \int \dfrac{\d}{\d t} Q_t^p \varphi(y) \, \d \mu_t(y) \qquad a.e.\ t\, .
\]
Then  Proposition \ref{prop: differential inequality of hopf lax with local steep} implies  that   $r(t) \geq \int u_p (\| v_t\|) \d \mu_t$ for a.e. $t \in [0,1]$, hence \eqref{eq:delta01} gives
\begin{align*}
\int\varphi^{c_p}_E  \d  {\mu}_1  - \int  \varphi \d {\mu}_0  & = \delta (1) - \delta (0)  \geq \int_0^1 \int u_p(\|v_t\|_g) \, \d \mu_t \, \d t .
\end{align*}
Taking the infimum among all $\varphi:E\to M$ steep and continuous, the conclusion \eqref{eq:claimperB} follows from Theorem \ref{thm: kantorovich duality}.
\end{proof}

\subsection{The Benamou--Brenier formula on spacetimes}

We begin by recalling existence  of $\ell_p$-geodesics with constant speed connecting two given probability measures. 
This result is well known in the Riemannian setting, see for instance \cite[Theorem 2.2.2]{Gig11} or \cite{AmbrGigSav}. 
It can also be found in the Lorentzian setting, e.g., \cite[Proposition 2.47]{Octet}.

\begin{lemma}
\label{existsence of curves of constant speed}
Let $(M,g)$ be a globally hyperbolic spacetime and let $0 \neq p <1$. 
Let $\mu_0, \mu_1 \in \Prob(M)$ with compact support such that $\mu_0 \preceq \mu_1$. Then there exists a continuous causal path $t \ni [0,1] \mapsto \mu_t \in \Prob (M)$, such that 
\[
u_p( | \dot \mu_t |_p) = u_p(\ell_p (\mu_0, \mu_1))
\]
for every $t \in (0,1)$, that is to say it is a constant speed $\ell_p$-geodesic connecting $\mu_0$ and $\mu_1$ and $\supp(\mu_t) \subset J(\supp(\mu_0), \supp(\mu_1))$ for every $t \in [0,1]$.
\end{lemma}

Now we are able to prove the Lorentzian Benamou--Brenier formula for optimal transport.

\begin{theorem}[Benamou--Brenier formula for spacetimes]
\label{thm: BB theorem Lor} 
Let $(M,g)$ be a globally hyperbolic spacetime and let $0 \neq p <1$. 
Let $\overline{\mu}_0, \overline{\mu}_1 \in \Prob(M)$ be two probability measures with compact support. 
Then
\begin{equation} 
\label{eq: BB formula}
u_p(\ell_p(\overline{\mu}_0, \overline{\mu}_1)) = \sup \int_0^1 \int u_p(\|v_t(x) \|_g) \, \d \mu_t(x) \, \d t \, ,
\end{equation}
where the supremum is taken over all solutions $(\mu_t,v_t)$ of the causal continuity inequality such that  $\overline{\mu}_0 \preceq \mu_t \preceq \overline{\mu}_1$ for every $t\in[0,1]$.
 Moreover, if $\overline{\mu}_0 \preceq \overline{\mu}_1$ then the supremum on the right hand side of \eqref{eq: BB formula}
is achieved.
\end{theorem}

\begin{proof}
First of all notice that if $\overline{\mu}_0 \not \preceq \overline{\mu}_1$ then \eqref{eq: BB formula} is verified since both sides are equal to $- \infty$, indeed there are no causal paths connecting $\overline{\mu}_0$ to $\overline{\mu}_1$. So, without loss of generality we assume that $\overline{\mu}_0 \preceq \overline{\mu}_1$. We are going to prove \eqref{eq: BB formula} by showing the two inequalities separately. 
Notice that the inequality `$\geq$' is proven in Proposition \ref{part B}. Now we are left to show the other inequality. 
Using Lemma \ref{existsence of curves of constant speed}, we find a continuous causal path $[0,1] \ni t \mapsto \mu_t \in \Prob(M)$ such that $\mu_0 = \overline{\mu}_0$ and $\mu_1=\overline{\mu}_1$ with speed $u_p( | \dot \mu_t |_p ) = u_p(\ell_p(\mu_0, \mu_1))$. If $\int_0^1u_p\left ( | \dot \mu_t |_p \right )\,\d t= -\infty$ the remaining inequality is trivial. If $\int_0^1u_p\left ( | \dot \mu_t |_p \right )\,\d t > -\infty$, via Proposition \ref{prop: part (A) in glob hyper}, we construct the vector fields $v_t$ associated to the curve $\mu_t$, giving
\[
u_p(\ell_p(\mu_0, \mu_1)) =\int_0^1 u_p(| \dot \mu_t |_p) \, \d t \leq \int_0^1 \int u_p(\| v_t(x) \|_g) \d \mu_t(x) \, \d t \, .
\]
Since `$\geq$' is already proved this suffices to conclude the claim. Moreover, in the case $\overline{\mu}_0 \preceq \overline{\mu}_1$, the $\ell_p$-geodesic coupled with $v_t$ maximizes the problem on the right hand side of \eqref{eq: BB formula} .
\end{proof}

In the following remark, we discuss some properties of a maximizing couple $(\mu_t,v_t)$ on the dynamical problem in \eqref{eq: BB formula}.

\begin{remark}
Let $(M,g)$ be a globally hyperbolic spacetime and let $0 \neq p <1$.  Let $\overline{\mu}_0 \preceq \overline{\mu}_1$ be two probability measures with compact support. 
Suppose that the $\ell_p$-geodesic $[0,1] \ni t \mapsto \mu_t$ constructed in Lemma \ref{existsence of curves of constant speed} satisfies $\int_0^1u_p\left ( | \dot \mu_t |_p \right )\,\d t > - \infty$ and let $\boldsymbol{\pi}$ be a lifting   and $v_t$ be the Borel vector fields as in \eqref{eq: barycenter}. \\ 
Then the couple $(\mu_t, v_t)$ is a maximizer of the supremum problem on the right hand side of \emph{\eqref{eq: BB formula}} such that
\begin{itemize}
\item[i)] $t \mapsto \mu_t$ is a continuous causal path and $\mu_0= \overline{\mu}_0$, $\mu_1 = \overline{\mu}_1$,
\item[ii)] $(\mu_t, v_t)$ satisfies equality in the causal continuity inequality \emph{\eqref{eq: causal continuity inequality}}.
\end{itemize} 
The first part follows by construction. We are left to prove ii). Given a smooth causal function $\varphi \in C^{\infty}(M)$, define the non-decreasing function $\Phi(t) = \int \varphi \, \d \mu_t$. Since \eqref{eq: causal continuity inequality} holds we have
\begin{align} \label{chain1}
\Phi(1) - \Phi(0) & \geq \int_0^1 \Phi'(t) \d t \geq \int_0^1 \d\phi_{x} (v_t(x)) \d \mu_t (x) = \int_0^1 \int \d \varphi_{\gamma_t}( \dot \gamma _t) \d \boldsymbol{\pi}(\gamma) \d t \, ,
\end{align}
where the last identity follows by the definition of $v_t$. Moreover, since $\boldsymbol{\pi}$ is concentrated on smooth geodesics, we can compute using the Fundamental Theorem of Calculus that
\begin{align*}
\Phi(1) - \Phi(0)= \int \varphi(\gamma_1) - \varphi (\gamma_0) \, \d \boldsymbol{\pi} (\gamma) = \int \int_0^1 \d \varphi_{\gamma_t} ( \dot \gamma_t)\d t \d \boldsymbol{\pi} (\gamma) \, .
\end{align*}
In particular, all the inequalities of \eqref{chain1} are equalities and so $\Phi$ is absolutely continuous and equality holds in the causal continuity inequality.
\hfill$\blacksquare$
\end{remark}

%\bibliography{bibliography} 
%\bibliographystyle{abbrv}

\end{document}